\documentclass[a4paper,12pt,oneside,reqno]{amsart}
\usepackage{hyperref}
\usepackage[headinclude,DIV13]{typearea}
\areaset{15.1cm}{25.0cm}
\usepackage{amsfonts,amssymb,amsmath,amsthm,bbm}
\usepackage{cancel} 
\usepackage[latin1] {inputenc}
\usepackage{graphicx, psfrag}

\newtheorem{theorem}{Theorem}[section]
\newtheorem{lemma}[theorem]{Lemma}
\newtheorem{proposition}[theorem]{Proposition}
\newtheorem{corollary}[theorem]{Corollary}

\theoremstyle{definition}

\theoremstyle{remark}
\newtheorem*{remark}{Remark}

\def\paragraph#1{\noindent \textbf{#1}}

\numberwithin{equation}{section}

\def\d{\mathrm{d}}

\def\<{\langle}
\def\>{\rangle}
\def\a{\alpha}

\def\e{\epsilon}

\def\k{\kappa}

\def\R{{\Bbb R}}  
\def\N{{\Bbb N}}  
\def\Z{{\Bbb Z}}  
\def\Q{{\Bbb Q}}  
\def\C{{\Bbb C}}  

\def\H{{\Bbb H}}

\let\cal=\mathcal

 \def \k {{\kappa}}
 
 \def \e {{\varepsilon}}

 \def \d {{\delta}}
 \def \a {{\alpha}}

 \def \ba {\begin{array}}
 \def \ea {\end{array}}

 \newcommand{\be}{\begin{equation}}
 \newcommand{\ee}{\end{equation}}

\newcommand{\bea}{\begin{eqnarray}}
 \newcommand{\eea}{\end{eqnarray}}
\def\TH(#1){\label{#1}}\def\thv(#1){\ref{#1}}
\def\Eq(#1){\label{#1}}\def\eqv(#1){(\ref{#1})}

 \def \1{\mathbbm{1}}

\begin{document}

 \title[Schmidt Games and Conditions on Resonant Sets]
{Schmidt Games and Conditions on Resonant Sets}

\author[S. Weil]{Steffen Weil}
\address{S. Weil\\ Institut f\"ur Mathematik der Universit\"at Z\"urich\\}
\email{steffen.weil@math.uzh.ch}

\date{\today}


\begin{abstract} 
Winning sets of Schmidt's game enjoy a remarkable rigidity.
Therefore, this game (and modifications of it) have been applied to many examples of complete metric spaces $(X,d)$ 
to show that the set of 'badly approximable points' \textbf{Bad}$(\cal{F})$, 
with respect to a given family $\cal{F}$ of resonant sets  in $X$,
is a winning set.
For these examples, strategies were deduced that are, in most cases, strongly adapted to the specific dynamics and properties of
the underlying setting.
We introduce a new modification of Schmidt's game which combines and generalizes the ones of \cite{KleinbockWeiss} and  \cite{McMullen}.
We then axiomatize conditions on the collection of resonant sets under which we can show \textbf{Bad}$(\cal{F})$ to be a winning set for the modification.
Moreover, we discuss properties of winning sets of this modification
and verify our conditions for several examples - 
among them, the set \textbf{Bad}$^{\bar r}$ of badly approximable vectors in $\R^n$, $\C^2$ and $\Z_p^2$, intersected with 'nice fractal sets', with weights $\bar r$
and, as a main example, the set of geodesic rays in a proper geodesic CAT(-1) space which avoid a suitable collection of convex subsets.
\end{abstract} 

\maketitle

\section{Introduction and Main Result}

\subsection{Introduction}
We begin with a motivation.
Let $(\bar X,d)$ be a metric space, $\mu$ a Borel probability measure and $T: \bar X \to \bar X$ an ergodic measure-preserving transformation.
Let $A\subset \bar X$ be a set of positive $\mu$-measure. 
Then, for $\mu$-almost every point $x\in \bar X$, the orbit of $x$ hits $A$ infinitely many times. 
The \emph{shrinking target problem}, due to Hill and Velani \cite{HillVelani}, 
considers sets shrinking in time.
More precisely, one considers a  sequence of nested measurable sets $A_n \subset X$ 
and is interested in the properties of the points in $ X$ whose orbit hits $A_n$ for infinitely many times $n$.
Such points are called \emph{well approximable} in analogy with Diophantine approximation.

For instance, identify the one point compactification $\bar \R = \R \cup \{\infty\}$ 
with the unit tangent space at a suitable point of the modular surface $\H^2/SL_2(\Z)$.
Then, the well approximable real numbers in the classical sense 
correspond to geodesics which enter a shrinking neighborhood of the only cusp of $\H^2/SL_2(\Z)$ infinitely often.
This is a set of full Lebesgue-measure.
Conversely, a badly approximable real number corresponds to a geodesic which avoids (i.e. does not enter) a certain neighborhood of the cusp.
The set of badly aproximable numbers is of Lebesgue-measure zero, yet of full Hausdorff-dimension and in fact a winning set for Schmidt's game.

Considering the lifts of the cusp neighborhood of  $\H^2/SL_2(\Z)$ to $\H^2$ - or rather their shadows in $\bar \R$ with respect to a given base point -
this motivates the following question.
Given a countable index set $\Lambda$, consider a family of sets $\{R_{\lambda} \subset \bar X : \lambda \in \Lambda\}$,
called \emph{resonant sets}, together with a family of contractions 
$\{ \psi_{\lambda} :  \R^+ \to \bar X : \lambda \in \Lambda\}$, 
where $R_{\lambda} \subset \psi_{\lambda}(t+s) \subset \psi_{\lambda}(t)$ for all  $t, s>0$.
Denote this family by $\cal{F}= (\Lambda, R_{\lambda}, \psi_{\lambda})$.
Given moreover a subset $X\subset \bar X$, 
define the set of  \emph{badly approximable points} in $X$ with respect to the family $\cal{F}$ by
\be
\nonumber
	\textbf{Bad}_X(\cal{F}) \equiv \{x\in X : \exists \ c =c(x) < \infty \text{ such that } x  \not \in \bigcup_{\lambda \in \Lambda} \psi_{\lambda}(c)  \};
\ee
that is the set of points $x$ in $X$ which are not contained in the uniformly shrinked neighborhoods, 
depending on the \emph{approximation constant} $c(x)$, of the family $\cal{F}$.
In this paper, we are interested in the question on what properties the set $ \textbf{Bad}_X(\cal{F})$ admits.

In a suitable framework, Kristensen, Thorn, Velani \cite{KristensenEtAl} already showed that $\textbf{Bad}_X(\cal{F})$ is of 'full' Hausdorff-dimension, that is the one of the space $X$.
We want to strengthen this result in two, somewhat 'orthogonal', directions.
On the one hand, non-trivial bounds on the Hausdorff-dimension of the set of badly approximable points with respect to a given upper bound on the approximation constants are determined in the author's work \cite{Weil3}.
On the other hand, in this paper, we use a different approach via modified Schmidt games where, at least in a reasonably nice setting, 
full Hausdorff-dimension is a property of winning sets of these games (among others, see Subsection \ref{SchmidtGame}).
In fact, winning sets of Schmidt's game (and modifications of it, called Schmidt games) enjoy a remarkable rigidity 
which has been exploited by many authors. 
This can be seen from the list 
\cite{An,Aravinda,AravindaLeuzinger,BBFKW,BFK,BroderickEtAl,Dani,Dani2,DaniShah,EinsiedlerTseng,Fishman,FishmanSimonsUrbanski,KB2,KleinbockWeiss,MasurEtAl,McMullen,Nesharim, Schmidt,Tseng}. 
However, in most cases strategies are deduced which are strongly adapted to the specific 
dynamics and properties of the considered example.
The purpose of this paper is the following.
Firstly, we introduce a modification of Schmidt's game 
which combines and generalizes  the ones of Kleinbock, Weiss \cite{KleinbockWeiss} and
McMullen \cite{McMullen} (as well as Broderick et al. \cite{BroderickEtAl}). 
Secondly, we abstractize conditions on a given collection of resonant sets and on the metric space $X$ in $\bar X$,
under which we determine explicit winning strategies with respect to the set $\textbf{Bad}_X(\cal{F})$ for this modified game.
Thirdly, we verify our conditions and obtain new or improve several known examples and results.

We emphasize that Schmidt's game remains a 'technique' since the obtained axiomatization guaranteeing a winning strategy is, 
of course, \emph{not} applicable to every example.
Nevertheless, confirmed by the applications in Section \ref{Apps}, at least  in appropriate settings it yields a significant simplification of the proofs and, 
by focussing on the conditions rather than determining a winning strategy, leads to new results.
Moreover, we point out that in the Euclidean setting, when $\bar X=\R^n$ is the Euclidean space, 
already Dani in \cite{Dani, Dani2}, Dani and Shah  \cite{DaniShah},
 as well as Fishman \cite{Fishman} deduced conditions under which \textbf{Bad}$_X(\cal{F})$ is a winning set.
Their conditions - as well as ours, compare with \eqref{Condition1} and \eqref{Condition2} below - 
concern mainly the (local) structure and distribution of both, the space $X$ in $\bar X$ and the resonant sets; 
for the precise statements see Theorem 3.2 in \cite{Dani}  and Theorem 2.2 in \cite{Fishman}. 

\subsection{Illustration of the main result}
We now give a first version of our main result, where we restrict to the 'standard' contractions defined below.

More precisely, let $(\bar X, d)$ be a proper metric space.
Fix $\sigma>0$ and $t_*\in \R$.
For a countable index set $\Lambda$,
let $\{R_{\lambda} \subset \bar X: \lambda \in \Lambda\}$ be a collection of resonant sets, 
where to each $R_{\lambda}$ we assign a \emph{size} $s_{\lambda}\geq t_*$ (also called  \emph{height}),
which determines the contraction
\be
\nonumber
	\psi_{\lambda}(c) \equiv \cal{N}_{e^{- \sigma (s_{\lambda}+c)}}(R_{\lambda}), \ \ \ c\geq 0. 
\footnote{ Here and in the following, given a metric space $\bar X$, $B(x,r) \equiv \{ y \in \bar X: d(x,y)\leq r\}$, $r>0$, is the closed ball around $x\in \bar X$
and $\cal{N}_{\e}(A) \equiv \cup_{x \in A}B(x,r)$ is the $\e$-neighborhood of a set $A \subset \bar X$.
}
\ee
Suppose that the resonant sets are \emph{nested} with respect to their sizes, that is, if $s_{\lambda}\leq s_{\beta}$ then $R_{\lambda} \subset R_{\beta}$, 
and that the sizes $\{s_{\lambda} \}\subset (t_*, \infty)$ are discrete.

Given a collection $\cal{S} \subset \cal{P}(\bar X)$ of subsets of $\bar X$, consider the following conditions
on a closed  subset $X\subset \bar X$ and the family $\cal{F} = (\Lambda, R_{\lambda}, s_{\lambda})$.
Firstly, $X$ is called \emph{$b_*$-diffuse with respect to $\cal{S}$}, for some $b_*\geq 0$, 
if for $S \in \cal{S}$ and any closed metric ball $B(x, e^{- \sigma t})$, $x\in X$, $t>t_*$, 
there exists $y \in X$ such that
\be
\label{Condition1}
	B(y, e^{- \sigma (t + b_*)}) \subset B(x, e^{- \sigma t}) - \cal{N}_{e^{- \sigma(t + b_*)}}(S).
\ee
When $\cal{S}$ denotes the collection of affine hyperplanes in  $\bar X = \R^n$, the definition is due to Broderick et al. \cite{BroderickEtAl} and $X$ is called \emph{hyperplane diffuse};
examples of such sets include the supports of absolutely decaying measures (see Subsection \ref{DiffuseSpaces} and \cite{LindenstraussEtAl}), in particular 'nice fractal sets' such as the Sierpinski gasket, Koch's curve or regular Cantor sets.
If $X$, with the induced metric, is  \emph{uniformly perfect} (see again Subsection \ref{DiffuseSpaces} for details), 
then it is $b_*$-diffuse with respect to the collection of points in $\bar X$;
for instance, the limit set of a non-elementary finitely generated Kleinian group is uniformly perfect \cite{JarviVuorinen}.

Secondly, the family $\cal{F}$ is \emph{locally contained in $\cal{S}$} 
if, given $B=B(x, e^{-\sigma t})$, $x \in X$, $t>t_*$ and $\lambda \in \Lambda$ with $s_{\lambda}\leq t$,%
\footnote{ Since the resonant sets are nested, it suffices for $\lambda_t \in \Lambda$ such that $s_{\lambda_t}$ is the maximal size with $s_{\lambda_t} \leq t$. }
there exists a set $S\in \cal{S}$
such that 
\be
\label{Condition2}
	B \cap R_{\lambda} \subset S.
\ee
Note that, if $\cal{S}$ is the collection of points in $\bar X$, 
then \eqref{Condition2} is satisfied if
for every $\lambda \in \Lambda$ and for any two distinct points $x$, $y \in R_{\lambda}$, 
we have
\be
\label{IntroductionDistinct}
	d(x,y)> 2 e^{-\sigma s_{\lambda}}.
\ee

Under these conditions we can determine an explicit winning strategy and show the following;  
for the proof, see Theorem \ref{Winning} and Theorem \ref{DiffuseProposition}.

\begin{theorem}
\label{IntroductionThm}
Suppose that $X$  is a closed subset of a proper metric space $\bar X$ which is $b_*$-diffuse with respect to a collection $\cal{S}$ of subsets of $\bar X$.
Moreover, let  $\cal{F}=(\Lambda, R_{\lambda}, s_{\lambda})$ be a family as above  with nested resonant sets and discrete sizes
such that $\cal{F}$ is locally contained in $\cal{S}$.
Then,  \textbf{Bad}$_X(\cal{F})$ is a winning set for Schmidt's game.
\end{theorem}

\noindent  Note that in this setting, \textbf{Bad}$_X(\cal{F})$ is in fact \emph{absolute winning with respect to $\cal{S}$}, a winning set for a modified game (see Subsection \ref{ModifiedGame} for the definition and details).

Among the examples from Section \ref{Apps}, the above theorem already applies to and simplifies the proofs of the following ones 
that will be discussed in more detail and in greater generality.
First, for $k\in \Lambda \equiv \N_{\geq 2}$ we define the set of rational vectors $R_k$ with size $s_k$ by  
\be
\nonumber
	R_{k} \equiv \{ \bar p/q : \bar p \in \Z^n, 0<q <k \}, \ \ \ \  s_k \equiv \log(k) + \log(n!\cdot 2^n),
\ee
which gives a nested and discrete family $\cal{F}$.
It is readily checked that, for $\sigma= 1+1/n$, \textbf{Bad}$_X(\cal{F})$ equals the set of badly approximable vectors 
\textbf{Bad}$^n_{\R^n}$ in a subset $X$ of $\R^n$ (see Subsection \ref{BadN} for details).
The Simplex Lemma (see Lemma \ref{SimplexLemma}) implies that, given $B=B(x, e^{-(1+1/n) t})$ and $k \in \N$ with $s_k \leq t$,
then $B \cap R_{k}$ is contained in an affine hyperplane.
Hence, Theorem \ref{IntroductionThm} shows that  \textbf{Bad}$^n_{\R^n} \cap X$ is a winning set for Schmidt's game 
for any hyperplane diffuse set $X \subset \R^n$.  
In particular,  \textbf{Bad}$^n_{\R^n} \cap X$ is hyperplane  absolute winning (see \cite{BroderickEtAl}) and, 
if $n=1$, then \eqref{IntroductionDistinct} holds so that \textbf{Bad}$^1_{\R} \cap X$ is a winning set for McMullen's game.
\\
Similar arguments apply to the sets of badly approximable vectors in $\R^n$, $\C^2$, $\Z^2_p$ respectively with weights for the modified game,
achieving new results (see Section \ref{Apps}).

Second, given a countable collection of pairwise disjoint horoballs $H_l$ in the real hyperbolic upper half space $\H^{n+1}$ 
tangent to the points $x_l \in \R^n$ and of Euclidean radius $1 \geq r_l >0$,%
\footnote{
For the example from the motivation, consider the collection of pairwise disjoint horoballs given by the collection  $B_{p/q} \subset \H^2$, 
where $B_{p/q}$ is an Euclidean ball tangent to $p/q\in \Q$ (with $p$, $q$ coprime) of radius $1/2 q^{-2}$.
Note that every such ball is a cover of the standard cusp neighborhood of the modular surface $\H^2/SL_2(\Z)$.
Moreover, we have $\textbf{Bad}_{\R}(\cal{F}) =  \textbf{Bad}_\R^1$.
}
define for $k\in \Lambda \equiv \N$,
\be
\nonumber
	R_{k} \equiv \{x_l \in \R^n: r_l \geq e^{-k} \}, \ \ \ \  s_k \equiv k + \log(2),
\ee
which again gives a nested and discrete family $\cal{F}$ in $\bar X=\R^n$.
Clearly, the disjointness of the horoballs shows that \eqref{IntroductionDistinct} is satisfied for $\sigma =1$.
Thus, for any uniformly perfect set $X$ in $\R^n$, Theorem \ref{IntroductionThm} implies that  \textbf{Bad}$_X(\cal{F})$ is a winning set for Schmidt's and in fact for McMullen's game.
Moreover,
\textbf{Bad}$_X(\cal{F})$ corresponds to the set of vertical geodesic lines in  $\H^{n+1}$ with endpoints in $ X$, 
for each of which the sequence of penetration lengths in the horoballs $H_l$ is bounded 
or, in other words, avoids the same collection of uniformly shrinked horoballs
(for further details and background, see Subsection \ref{CAT(-1)}). 
This already simplifies and shortens the proof of McMullen (compare with \cite{McMullen}) significantly.
\\
Similarly, as a main example, we consider the set of geodesic rays avoiding a suitable collection of convex sets  in a proper geodesic CAT(-1)-space, such as a collection of geodesic lines or even 'higher-dimensional' subspaces, which again achieves new results.
\\

\noindent \emph{Outline of the paper.}
In Section \ref{SectionSchmidtGames}, we first recall the $\psi$-modified Schmidt game due to \cite{KleinbockWeiss} and its properties (Subsection \ref{SchmidtGame}).
We introduce our modified version of the game in this setting and deduce properties of winning sets for this game (Subsection \ref{ModifiedGame}).
Moreover, we consider different conditions on the collection of resonant sets and on the metric space under which 
the set of badly approximable points is a winning set for the respective versions of the game (Subsection \ref{Strategy}).
Finally, we discuss  diffusion properties of the space $X$, consider suitable (absolutely decaying) measures supported on $X$, 
and,  the structure and distribution of the resonant sets 
under which the deduced conditions are satisfied (Subsection \ref{DiffuseSpaces}).

In Section \ref{Apps}, we verify the conditions for several examples, where we distinguish between examples coming from number theory and the ones coming from dynamical systems:
For the first part, we consider the set of badly approximable vectors in $\R^n$, $\C^2$ and $\Z_p^2$ with weights (see Subsections \ref{BadRN}, \ref{BadCn}, \ref{BadZn} respectively).
For the second part, we consider the set of sequences in the Bernoulli-shift which avoid periodic sequences (Subsection \ref{Bernoulli}) and the set of orbits of a sequence of matrices avoiding a sequence of separated sets (Subsection \ref{ToralEndo}).
Moreover, in more detail, we consider the set of geodesics in a proper geodesic CAT(-1)-space 
which avoid certain convex subsets such as a collection of disjoint horoballs or neighborhoods of geodesic lines or of a separated set (see Subsection \ref{CAT(-1)}).
\\

\noindent \emph{Acknowledgment.}
The author wants to thank his advisor, Viktor Schroeder, for helpful discussions, comments and for introducing him to Schmidt's game.
Moreover, he is grateful to Jouni Parkkonen and Jean-Claude Picaud for their questions which furtherly aroused his interest in this subject.
In particular, he wants to express his gratitude to Dmitry Kleinbock and Barak Weiss for many helpful suggestions which led to further results and improvements.
Finally, the author acknowledges the support by the Swiss National Science Foundation (Grant: 135091).

\section{Schmidt Games on Paramater Spaces}
\label{SectionSchmidtGames}

In this section, we combine two versions of Schmidt's game due to  \cite{KleinbockWeiss} and  \cite{McMullen} 
in order to introduce a new modification.
We first introduce but modify the setting of this section which is the notion of \cite{KleinbockWeiss}.
Let $(X,d)$ be a complete metric space.
Fix $t_* \in \R \cup \{- \infty\}$ and define $\Omega \equiv X \times (t_*, \infty)$, the set of \emph{formal balls} in $X$.
Let $\cal{C}(X)$ be the set of nonempty compact  subsets of $X$
and assume we are given a function $\psi : \Omega  \to \cal{C}(X)$
such that, for all $(x, t) \in \Omega$ and for all $s\geq 0$, we have
\be
\label{Mono}
	\psi(x,t+s) \subset \psi(x,t).
\ee
We can hence view $\Omega$ as parameter space for the function $\psi$ which we call \emph{monotonic}.

For instance, if $X$ is proper, set $t_* = - \infty$ and 
for $x\in X$, $r>0$, let  $B(x,r) \equiv \{y\in X: d(x,y)\leq r\} \in \cal{C}(X)$.
For $\sigma>0$, the \emph{standard function $\bar \psi_{\sigma} \equiv B_{\sigma}$} is given by the monotonic function 
\be
\label{Standard}
	B_{\sigma} : X \times (-\infty, \infty) \to \cal{C}(X), \ \ \ \ B_{\sigma} (x,t) \equiv B(x, e^{- \sigma t}).
\ee

Moreover, for a subset $Y\subset  X$ and  $t>t_*$, we call $(Y,t) \equiv \{(y, t) : y\in Y\} $ a \emph{formal neighborhood},
and define $\cal{P}= \cal{P}( X) \times (t_*, \infty) $ to be the set of formal neighborhoods.
Define the \emph{$ \psi$-neighborhood} of $(Y,t) \in \cal{P}$ by
\be
\nonumber
	\psi(Y,t) \equiv \bigcup_{y \in Y}\psi (y,t). 
\ee
Note that by \eqref{Mono}, $\psi(Y,t + s ) \subset \psi(Y,t)$ for all $s \geq 0$.

\subsection{The $\psi$-modified Schmidt game} 
\label{SchmidtGame}
We recall the $(\psi,a_*)$-modified Schmidt game due to \cite{KleinbockWeiss}, where $a_*\geq 0$.
Two players, $A$ and $B$, pick numbers $a$ and $b$ both bigger than $a_*$.
Player $B$ starts with his first move by choosing a formal ball $\omega_1 = (x_1, t) \in \Omega$.
Given a choice $\omega_k=(x_k, t_k)$ of $B$, due to \eqref{Mono},
player $A$ can (and must) choose a formal ball $\bar \omega_k = (\bar x_k, t_k+a) \in \Omega$ 
such that $\psi(\bar \omega_k)\subset \psi(\omega_k)$.
Also player $B$ continues by choosing a formal ball $ \omega_{k+1} = (x_{k+1}, t_k+a +b) \in \Omega$ 
such that $\psi(\omega_{k+1}) \subset \psi( \bar \omega_k)$.
The game continues in this manner and we obtain  a nested sequence of compact sets
\be
\nonumber
	B_1 \equiv \psi(\omega_1) \supset A_1 \equiv \psi (\bar \omega_1) \supset B_2 \equiv \psi (\omega_2) \supset \dots 
	\supset B_k \equiv \psi( \omega_k) \supset A_k \equiv \psi(\bar \omega_k) \supset \dots, 
\ee
where $\omega_k= (x_k, t_k)$ and $\bar \omega_k= (\bar x_k, \bar t_k)$ satisfy
\be
\nonumber
	t_k= t_1+(k-1)(a+b), \text{ and }\ \ \ \bar t_k= t_1 + (k-1)(a+b)+a.
\ee
The intersection of compact nested sets, given by
\be
\nonumber
	\bigcap_{k=1}^{\infty} B_k = \bigcap_{k=1}^{\infty} A_k,
\ee 
is nonempty and compact.
A given subset $S\subset X$ is called \emph{$(\psi, a_*, a,b)$-winning}, 
if player $A$ can find a \emph{strategy} which guarantees that $\cap_{k\geq1} B_k$ intersects $S$,  
no matter what $B$'s choices are.
The set $S$ is called \emph{$(\psi, a_*, a)$-winning} if $S$ is \emph{$(\psi, a_*, a,b)$-winning} for every $b>a_*$.
$S$ is \emph{$(\psi, a_*)$-winning} if it is \emph{$(\psi, a_*, a)$-winning} for some $a>a_*$
and \emph{$\psi$-winning} if it is $(\psi, a_*)$-winning for some $a_*\geq 0$.

\begin{remark}
Note that if $S$ is a $B_{\sigma}$-winning set, then it is also a $B_1$-winning set.
In fact, if $S$ is $(B_{\sigma}, a_*, a, b)$-winning, then it is $(B_{1}, \sigma a_*, \sigma a, \sigma b)$-winning.
\end{remark}

With respect to the standard monotonic function $\psi=B_1$,
the game described above coincides with the original $(\alpha, \beta)$-Schmidt game
for the choice
\be
\nonumber
	a= - \log(\alpha), \ \ \ b= - \log(\beta), \ \ \ a_*=0, \ \ \ t_*= - \infty.
\ee
If moreover $X=\R^n$ is the Euclidean space, 
then a winning set $S$ enjoys the following properties (see \cite{BroderickEtAl,Dani2,Schmidt}).
\begin{itemize}
    \item[1.] A winning set is dense and thick; a subset $Y$ of a metric space $X$ is \emph{thick} if for any nonempty open set $U \subset X$, 
    							$Y\cap U$ has full Hausdorff-dimension,
    \item[2.] a countable intersection of  $\alpha$-winning sets is $\alpha$-winning,
    \item[3.] winning sets are preserved by bi-Lipschitz homeomorphisms, and,
    \item[4.]  winning sets are \emph{incompressible}; 
    			that is, given a nonempty open set $U \subset \R^n$ and a countable sequence of uniformly%
			\footnote{ That is, the Lipschitz constants $L_i$ of $F_i$ are bounded.}
			bi-Lipschitz maps $F_i : U \to \R^n$, 
			then $\cap_{i=1}^\infty F_i^{-1}(S)$ has Hausdorff-dimension $n$.
\end{itemize}

Unfortunately, these properties are not satisfied in general; 
in fact, see \cite{KleinbockWeiss}, Proposition 5.2, for a $\psi$-winning set which is of Hausdorff-dimension zero in a space of positive dimension.
However, the following (and further) properties for the $\psi$-modified Schmidt game can be found in \cite{KleinbockWeiss}%
\footnote{ Note that \cite{KleinbockWeiss} uses a slightly different setting. Nevertheless, the properties hold true with the same arguments. }
.
\begin{itemize}
    \item[1.] 	Let $S_i \subset X$, $i \in \N$, be a sequence of $(\psi, a_*, a)$-winning sets. 
			Then, $\cap_{i\geq 1} S_i$ is also  $(\psi, a_*, a)$-winning. 
    \item[2.]		Let $\Omega_i= X_i \times (t_*,\infty)$, and $ \psi_i$ be given for $i=1,2$.
			Suppose that $S_i \subset X_i$ is a $(\psi_i, a_*)$-winning set for $i=1,2$.
			Then $S_1 \times S_2$ is a $(\psi_1 \times \psi_2, a_*)$-winning set in $X_1 \times X_2$ with the product metric,
			where $\psi_1 \times \psi_2( x_1, x_2, t) \equiv \psi_1(x_1, t) \times \psi_2(x_2, t)$.
\end{itemize}

\noindent Moroever, let $\mu$ be a locally finite Borel measure on $X$.
Denote by $O(x,r) \equiv \{y \in X: d(x,y)<r\}$ the open metric ball around $x$.
The \emph{lower pointwise dimension} of $\mu$ at $x\in $ supp$(\mu)$ is defined by%
\be
\nonumber
	d_{\mu}(x) \equiv \liminf_{r \to 0} \frac{\log( \mu(O(x,r) )}{\log r}.
\ee
For every open $U \subset X$ with $\mu(U)>0$,
\be
\nonumber
	d_{\mu}(U) \equiv \inf_{x \in U \cap \text{ supp}(\mu)} d_{\mu}(x),
\ee
which is known to be a lower bound for the Hausdorff-dimension of supp$(\mu) \cap U$ (see \cite{Falconer}, Proposition 4.9 (a)).
The measure $\mu$ is called \emph{Federer} if there are $K>0$ and $R>0$ such that for all $x\in $ supp$(\mu)$ and $0<r<R$,
\be
\nonumber
	\mu(O(x,3r)) \leq K \mu(O(x,r)).
\ee
In the case that we consider the standard function $\psi_1$, i.e., we focus on the classical Schmidt-game,
the following lower estimate on the Hausdorff-dimension is given.

\begin{proposition}[\cite{KleinbockWeiss}, Proposition 5.1]
\label{LowerDimension}
If $S$ is a winning set (in the sense of Schmidt) in a complete metric space $X$ which supports a Federer measure $\mu$ with $X=$ supp$(\mu)$,
then for every nonempty open set $U\subset X$, we have dim$(S\cap U) \geq d_{\mu}(U)$, where 'dim' stands for the Hausdorff-dimension.
\end{proposition}

If $\mu$ satisfies a \emph{power law}, that is, there exist $\delta$, $c_1$, $c_2$ and $R>0$ such that for every $0<r< R$ and $x\in $ supp$(\mu)$
we have
\be
\nonumber
	c_1 r^{\delta} \leq \mu(O(x,r)) \leq c_2 r^{\delta},
\ee
then $\mu$ is Federer and we have $d_{\mu}(x)= \delta$.

\subsection{The weak $\psi$-modified Schmidt game}
\label{ModifiedGame}
For $b_*>0$ consider the following modification of  rules for the players $A$ and $B$.
Fix a parameter $b\geq   b_*$.
Player $B$ starts again with a formal ball $\omega_1=(x_1, t_1) \in \Omega$.
Then, given a formal ball $\omega_k=(x_k, t_k)  \in \Omega$ of player $B$,
player $A$ must choose a nonempty set $\cal{L}_{b}^{\psi}(\omega_k) \subset \Omega$ of \emph{legal moves},
\be 
\label{LegalPlay2}
	\cal{L}_{b}(\omega_k)\equiv \{ \omega = (x, t_k + \bar b) : b_* \leq \bar b \leq m_k b, \     \psi(\omega) \subset \psi(\omega_k), \ \cal{C}(\omega_k)  \},
\ee
where $ \cal{C}(\omega_k)$ denotes possible further conditions which $A$ requires and $m_k \in \N$ is an integer which $A$ chooses at each step.
$B$ then chooses a formal ball $\omega_{k+1}\in  \cal{L}_{b}^{\psi} (\omega_k)$ and the game continues in this manner.  
Since $\psi(\omega_k) \supset \psi(\omega_{k+1})$, we obtain a nested sequence 
\be
 \label{ChoiceA}
\nonumber
	B_1  \supset B_2  \supset \dots \supset B_k \supset \dots,
\ee
where $B_k= \psi(x_k, t_k)$ satisfies condition $\cal{C}(\omega_k)$. 
If the nonempty compact set $\cap_{k\geq 1} B_k$ intersects a given set $S \subset X$, then $A$ \emph{wins} this game.
The set $S$ is called \emph{weakly $(\psi, b_*,b)$-winning} if player $A$ finds a strategy such that $A$ wins for every possible game, given the parameter $b$.
$S$ is called \emph{weakly $(\psi, b_*)$-winning} if it is  \emph{weakly $(\psi, b_*,b)$-winning} for every $b\geq b_*$
and \emph{weakly $\psi$-winning} if it is weakly $(\psi,b_*)$-winning for some $b_*> 0$.

\begin{remark}
Note that to leave $\cal{L}_{b}(\omega_k)$ nonempty is always possible by \eqref{Mono}. 
Moreover, the conditions that $\psi(\omega_{k+1}) \subset \psi(\omega_k)$ and  $\bar b \leq m_k b$ seemed to be the least suitable conditions to already assume for player $A$ (and for our purpose) 
but can of course be weakened as well.
The requirement that $b_*>0$ implies that $t_k \to \infty$ which can be avoided if we say that $A$ wins when $t_k \not \to \infty$.%
\footnote{ This alternative rule was chosen, for instance, by \cite{FishmanSimonsUrbanski}. Note that, if $S$ is dense in $X$ and $\psi = B_1$, then $B$ looses as soon as $t_k \not \to \infty$. }
\end{remark}

The difference to the original $\psi$-modified Schmidt game is that, rather than forcing $B$ in a certain direction,
 $A$ can precisely determine $B$'s choices in the next move. 
Since $A$ might leave $B$ only one choice in each step, the weak $\psi$-modified Schmidt game loses in some sense the character of a game.
Moreover, the conditions $C(\omega_k)$ determine the 'control' player $A$ chooses and the more conditions $A$ requires, the less properties $S$ might enjoy.
Therefore, player $A$ also has an interest in leaving $B$ as much choices and freedom as possible, with respect to a winning strategy.
In particular, we are interested in conditions on strategies for player $A$ such that a weakly $\psi$-winning set $S$ 
satisfies similar or even the same properties than winning sets for Schmidt's, McMullen's or the $\psi$-game.

We want to  point out the following special cases of modifications of Schmidt's game,
where, given a choice $\omega_k = (x_k, t_k) \in \Omega$ of $B$, $A$ chooses a set $A_k\subset X$ 
and requires for the condition $\cal{C}(\omega_k)$ that
\be
\label{ConditionInducedBySets}
	\psi(\omega)  \subset \psi(\omega_k) - A_k \ \  \text{ and } \  \ m_k=m_*=1.
\ee

First, let $\cal{S} = \{ S \subset X\}$ be a given collection of subsets of $\bar X$.
Assume then that for each of the sets $A_k$ is either empty or a $\psi$-neighborhood
\be
\label{HAW}
	A_k= \psi(S_k , t_k +a_k), \ \ S_k \in \cal{S}, \ \ a_k\geq b,
\ee
and call a winning set under these requirements \emph{absolute $\psi$-winning with respect to $\cal{S}$};
compare with   \cite{FishmanSimonsUrbanski} for the case that $\psi=B_1$ is the standard function.

Consider the standard case that 
\be
\nonumber
	X=\R^n, \ \ \  \psi=B_1, \ \ \ b_*= \log(3), \ \ \  t_*= - \infty .
\ee
Clearly, if $\cal{S}$ is the set of points in $\R^n$, this modification corresponds to the one of McMullen \cite{McMullen}, 
called \emph{absolute winning game} and a winning set is called \emph{absolute winning}.
Note that an absolute winning set in $\R^n$ is in particular a Schmidt winning set and in fact satisfies stronger properties (see \cite{McMullen}).

In the case that $\cal{S}$ denotes the set of affine hyperplanes in $\R^n$ (or in a vector space),
then this modification corresponds to the one of Broderick et al. \cite{BroderickEtAl}, 
called \emph{hyperplane absolute winning game}  and a winning set is called \emph{hyperplane absolute winning} (short HAW set). 
Again, note that a HAW-set in $\R^n$ is in particular a Schmidt winning set and in fact satisfies stronger properties (see \cite{BroderickEtAl}).

Second, let $b_* > a \geq a_* \geq 0$.
Assume the sets $A_k $ to be the complements of $\psi$-balls 
\be
\nonumber
\label{Schmidt}
	A_k=\psi(y_k, t_k + a)^C,  \ \ (y_k, t_k + a) \in \Omega \ \text{ with }\  \psi(y_k, t_k + a) \subset B_k = \psi(x_k, t_k),
\ee
If $\cal{C}(\omega_k)$ moreover requires that $\bar b= b$ in \eqref{LegalPlay2},
this modification corresponds to the $(\psi,a_*, a, b-a)$-game and in particular to Schmidt's game for $X=\R^n$ and $\psi=B_1$.

Now  in general, if $\cal{C}(\omega_k)$ requires 
 for all sets $A_k\subset X$ which $A$ chooses in \eqref{ConditionInducedBySets} that
there exists a formal ball $\bar \omega = (\bar x, t_k+ b_*)\in \Omega$ such that 
\be 
\label{LegalPlay3}
	\psi(\bar \omega) \subset \psi(\omega_k) - A_k,
\ee
then a weakly $\psi$-winning set is $\psi$-winning.

\begin{lemma}
\label{AbsWinWin}
If \eqref{LegalPlay3} is satisfied, then a weakly $(\psi, b_*)$-winning set $S$ is  $(\psi, a_*)$-winning for all $a_*\geq b_*$.
\end{lemma}
 
\begin{proof}  
Given $a \geq a_* \geq b_*$, $b>0$, set $\tilde b= a+b \geq b_*$.
Let player $A$ play the $(\psi,a_*,a,b)$-modified Schmidt game 
and consider a further player $\bar A$ who plays the weak $(\psi, b_*, \tilde b)$-modified Schmidt game.
Suppose that player $B$ has chosen his $k$-th move $\omega_k=(x_k, t_k) \in \Omega$. 
By \eqref{LegalPlay3},
$\bar A$ chooses a set $A_k \subset X$ such that 
there exists a formal ball $\bar \omega = (\bar x, t_k+ b_*) \in \Omega$ with
\be 
\nonumber
	\psi(\bar \omega) \subset \psi(\omega_k) -  A_k. 
\ee
By \eqref{Mono} and since $a\geq b_*$, 
there exists a formal ball $\bar \omega_{k+1} = (\bar x_{k+1}, t_k+a) \in \Omega$ such that $\psi(\bar \omega_{k+1}) \subset \psi( \bar \omega)$
which we take as $A$'s choice.
Note that any move $\omega_{k+1}=(x_{k+1}, t_k+k(a+b)) = (x_{k+1}, t_k+ \tilde b) \in \Omega$ such that 
$\psi(\omega_{k+1}) \subset \psi(\bar \omega_{k+1})$ 
of $B$ is a legal move for both games.
Since $\bar A$ has a weak winning strategy, we see that 
\be
\nonumber
	\bigcap_{k\geq 1} \psi (\bar \omega_k) = \bigcap_{k\geq 1} \psi ( \omega_k)  
\ee
intersects $S$. Hence, $A$ wins and $S$ is also a $(\psi, a_*, a,b)$-winning set.
\end{proof}

Hence, in view of the properties of $\psi$-winning sets (see Subsection \ref{SchmidtGame}), 
we will consider conditions which ensure that \eqref{LegalPlay3} is satisfied
so that that the weak $\psi$-modified Schmidt game is at least as strong as the  $\psi$-modified Schmidt game.
However, some of the properties of $\psi$-winning sets can still be true in the weaker setting.

In fact, let $S$ be a $(\psi, b_*, b)$-weakly-winning set.
In order to estimate the lower bound for the Hausdorff-dimension of $S$,
we consider the conditions given by \cite{KleinbockWeiss} and only need to modify ($\mu2$) below:
\begin{itemize}
\item[(MSG1)]	For any open set $\emptyset \neq U \subset X$ there is $\omega \in \Omega$ such that $\psi(\omega) \subset U$.
\item[(MSG2)]	There exist $C$, $\sigma>0$ such that diam$(\psi(x,t)) \leq Ce^{-\sigma t }$ for all $(x,t )\in \Omega$.
\end{itemize}
Note that if (MSG1) is satisfied, a weakly $\psi$-winning set is dense.
Let moreover $\mu$ be a locally finite Borel measure on $X$ such that:
\begin{itemize}
\item[($\mu1$)] For every formal ball $\omega \in \Omega$ we have $\mu(\psi(\omega))>0$.
\item[($\mu2$)] 
There exist constants $c=c( b)>0$ and $m_* = m_*(b)\in\N$ with the following property:
If $\omega_k \in \Omega$ 
is a choice of $B$ in the $(\psi, b_*,b)$-game, 
there exist legal moves $\omega^1_{k+1}, \dots ,  \omega^n_{k+1} \in \cal{L}_{b}(\omega_k)$, $\omega^i_{k+1}=(x^i_{k+1}, t_k + m_kb)$, $m_k=m_*$,
with respect to the $(\psi, b_*, b)$-strategy of $A$, 
which satisfy $\mu(\psi(\omega^i_{k+1})\cap \psi ( \omega^j_{k+1})) = 0 $ when $i\neq j$
as well as
\be
\label{mu2} 
	\mu \big( \bigcup_{i=1\dots n} \psi( \omega^i_{k+1}) \big) \geq c \cdot \mu(\psi(\omega_k)).
\ee
\end{itemize}
Note that from (MSG1) and ($\mu$1), $\mu$ must have full support, i.e. supp$(\mu) = X$.

\begin{proposition} 
\label{LowerDimension2}
Suppose that $X$, $\Omega$, $\psi$ and the measure $\mu$ satisfy (MSG1-2) and ($\mu$1-2)
with respect to a weakly $(\psi, b_*,b)$-winning set $S$.
Then for every nonempty  open set $U\subset X$ we have that 
\be
\nonumber
	\text{dim}(S\cap U) \geq d_{\mu}(U) + \frac{\log(c)}{\sigma m_* b},
\ee
where $\sigma$, $c=c( b)$ and  $m_*$ are the constants of (MSG2) and ($\mu2$). 
\end{proposition}

\begin{proof}
Similarly to the proof of \cite{KleinbockWeiss}, Theorem 2.7, 
one constructs a strongly treelike countable family of compact subsets of $X$ whose limit set $A_{\infty} \cap U$ is a subset of $S \cap U$.
We start with a formal ball $\omega_1 \in \Omega$ such that $\psi(\omega_1) \subset U$.
The difference is that, instead of using the choices of $A$, we use the choices of $B$ given in ($\mu2$) in
order to obtain that 
\be
\nonumber
	\text{dim}(A_{\infty} \cap U)
	\geq d_{\mu}(U) + \frac{\log(c)}{\sigma m_* b}.
\ee
The proof follows.
\end{proof}

\subsection{The framework, conditions on the resonant sets and  strategies}
\label{Strategy}
Let $\bar X$ be a proper metric space and $X$ a closed subset of $\bar X$ which is, with the induced metric, a complete metric space.
In many applications, we are interested in playing the $\psi$-game on $X$ 
but do not require the resonant sets to be contained in $X$ but in $\bar X$.
Therefore, let $\bar \Omega = \bar X \times (t_*, \infty)$ and $\Omega= X \times (t_*, \infty) \subset \bar \Omega$.
Let $\bar \psi : \bar \Omega \to \cal{C}(\bar X)$ be a monotonic function  on $\bar \Omega$,
which induces a monotonic function $\psi$ on $\Omega$,  defined by
\be
\nonumber
	\psi(\omega) \equiv \bar \psi(\omega) \cap X, \ \ \ \omega \in \Omega.
\ee

Now, let $\Lambda$ be a countable index set  and  $\{R_{\lambda} \subset \bar X : \lambda \in \Lambda\}$ be a family of \emph{resonant sets} in $\bar X$,
where we assign a \emph{size} $s_{\lambda}\geq s_*$ to every $R_{\lambda}$  with $t_*<s_*\in \R$.
We consider the contractions of the $(\bar \psi, s_{\lambda})$-neighborhoods of $R_{\lambda}$, that is
\be
\nonumber
\label{Contraction}
	\psi_{\lambda}(c) \equiv \bar \psi(R_{\lambda}, s_{\lambda} + c) \subset \bar \psi(R_{\lambda}, s_{\lambda} ), \ \ \ c\geq 0.
\ee
Denote this family by
\be	
\nonumber
\label{Family}
	\cal{F}= ( \Lambda, R_{\lambda}, s_{\lambda}).
\ee
Assume that the family $\cal{F}$  satisfies the following conditions.
\begin{itemize}
\item[(N)]
The resonant sets $\{R_{\lambda}\}$ are \emph{nested} with respect to their sizes, 
that is, for $\lambda, \beta \in \Lambda$ we have
\be
\nonumber
\label{Stable}
	s_{\lambda} \leq s_{\beta} \implies R_{\lambda} \subset R_{\beta}.
\ee
\item[(D)]
The sizes $\{s_{\lambda}\}$ are \emph{discrete}, 
that is, for all $t> t_*$ we have 
\be
\nonumber
\label{Discrete}
	 \lvert \{ \lambda \in \Lambda: s_{\lambda} \leq t\} \rvert < \infty.
\ee
\end{itemize}
We then define the set of \emph{badly approximable points} with respect to $\cal{F}$ by 
\be
\nonumber
	\textbf{Bad}_X^{\bar \psi}(\cal{F})= \{x\in X : \exists \ c =c(x)< \infty \text{ such that } x  \not \in \bigcup_{\lambda \in \Lambda} \psi_{\lambda}(c)  \},
\ee
or simply by $\textbf{Bad}(\cal{F})$ if there is no confusion about the parameter spaces under consideration.

Using (N) and (D), we define a 'one-parameter' family of resonant sets and sizes as follows.
For a parameter $t\geq s_1$, let $\lambda_t \in \Lambda$ such that $s_t \equiv s_{\lambda_t}$, called \emph{relevant size}, is the maximal size with $s_{\lambda} \leq t$.
We define the \emph{relevant resonant set} with respect to the parameter $t$ by
\be
\nonumber
	R(t) \equiv \bigcup_{ s_{\lambda} \leq t} R_{\lambda} = R_{\lambda_t}.
\ee
Moreover, for $t\geq s_1$ and $b> 0$, we let 
\be
\label{RB}
	R(t,b) \equiv 	R( t) - R(t-b) 
\ee 
be the set of resonant points for which the 'minimal size' belongs to the spectrum $(r-b, r]$.


For $b_*>0$, $n_*\in \N$ and $L_*\geq 0$, we consider two conditions, a strong and a weak one, on the space $X$ 
and a nested and discrete family $\cal{F}$.
\begin{itemize}
\item[($b_*$)]
$(\Omega, \psi)$ is \emph{strongly $b_*$-diffuse with respect to the family $\cal{F}$},
if there exists  $n \in \N$ such that 
for all formal balls $\omega=(x, t) \in \Omega$
there exists a formal ball $\omega' = (x', t+b_*) \in \Omega$ such that
\be
\label{bDiffuse}
	\psi (\omega') \subset \psi(\omega) - \bar \psi (R(t), t+ nb_*).
\ee
\item[($b_*, n_*, L_*$)]
$(\Omega, \psi)$ is \emph{$(b_*,n_*, L_*)$-diffuse with respect to the family $\cal{F}$},
if for all $b> b_*$ there exists a $n=n(b) \in \N$ such that,
for all formal balls $\omega=(x, t)\in \Omega$ ,
there exists a formal ball $\omega' = (x', t+b) \in \Omega$ 
such that
\be
\label{BarBDiffuse}
	\psi (\omega') \subset \psi(\omega) - \bar \psi (R(t, n_*(b+L_*)), t+ nb).
\ee
\end{itemize}

Condition $(b_*)$ is too strong in general (see Subsection \ref{ToralEndo} and \ref{CAT(-1)}, Case 3.)
but implies  $(b_*,n_*, L_*)$ for all $n_*\in \N$, $L_* \geq 0$ and is sufficient to guarantee that if \textbf{Bad}$(\cal{F})$ is weakly $(\psi,b_*)$-winning
it is also $(\psi,b_*)$-winning by Lemma \ref{AbsWinWin}.
\\
In fact, under these conditions we can define the following strategies for the set 
\be
\nonumber
	S=\textbf{Bad}(\cal{F}).
\ee
Fix a parameter $b>b_*$
and assume $B$ chose the formal ball $\omega_1=(x_1, t_1) \in \Omega$.

\emph{ The strategy for player $A$ under the condition $(b_*, 1, 0)$.}
Let  $m_*\in \N$ be the minimal integer such that $ m_*b \geq t_1 -s_1 $ and let $l_*=n(m_* b)$ be as in \eqref{BarBDiffuse}.
Given the times $t_k$, define the relevant resonant sets $
	R_k \equiv R(t_{k} , m_* b)$. 
For $k  \geq 1$, assume that $B$ chose the formal ball $\omega_k=(x_k, t_k)\in \Omega$. 
Note that if we set
\be
\label{Strategy2}
	   A_k \equiv \bar  \psi (R_k, t_k + l_*( m_* b)) \cap X, 
\ee
then, by \eqref{BarBDiffuse}, there exists a formal ball $\omega'=(x'_k, t_{k}+ m_* b) \in \Omega$ such that 
\be
\label{Strategy1}
	\psi (\omega') \subset \psi(x_k, t_k) - \bar  \psi (R_k, t_k + l_*(m_* b)) = \psi(x_k, t_k) - A_k.
\ee
Thus,  we define the strategy of player $A$ to choose the nonempty set of legal moves
\be 
\label{InducedStrategy}
	\cal{L}_{b}(\omega_{k})\equiv \{ \omega = (x, t_k + \bar b) :  b_* \leq \bar b \leq m_* b, \    \psi(x, t_k + \bar b) \subset \psi(\omega_k) - A_k  \}.
\ee

 \emph{The strategy for player $A$ under the condition $(b_*)$.}
Let now $R_k= R(t_k)$, $m_*=1$ and $l_* = n(b_*)$ as in \eqref{bDiffuse} and set
\be
\label{StrongStrategy1}
	   A_k \equiv \bar  \psi (R(t_k), t_k + l_* b_*) \cap X. 
\ee
We define the strategy of player $A$ with respect to $\omega_k$ to choose the  set of legal moves
\be 
\label{StrongInducedStrategy}
	\cal{L}_{b}(\omega_{k})\equiv \{ \omega = (x, t_k + \bar b) :  b_* \leq \bar b \leq  b, \    \psi(x, t_k + \bar b) \subset \psi(\omega_k) - A_k  \},
\ee
which is nonempty by \eqref{bDiffuse}.
 

With respect to these strategies, we show our first main result.

\begin{theorem}
\label{Winning}
Let $\cal{F}$ be a nested and discrete family.

If $(\Omega, \psi)$ is $(b_*,n_*, L_*)$-diffuse with respect to $\cal{F}$, 
then \eqref{InducedStrategy} defines a weakly $(\psi, b_*,b)$-winning strategy for the set \textbf{Bad}$(\cal{F})$.

If $(\Omega, \psi)$ is strongly $b_*$-diffuse with respect to $\cal{F}$,
then  \textbf{Bad}$(\cal{F})$ is in particular $(\psi, a_*)$-winning for every $a_*\geq b_*$.
\end{theorem}

\begin{proof}[Proof of Theorem \ref{Winning}]
We first show that the induced strategy is winning under the condition $(b_*, 1, 0)$.
Hence, let $x_0 \in \cap_{k \geq 1} \psi(\omega_k)$.
Assume that $x_0 \in \bar \psi(R_{\lambda_0}, s_{\lambda_0})$ for some $\lambda_0 \in \Lambda$ 
(if no such $\lambda_0$ exists, then $A$ has already won).
Since $t_1 - m_* b \leq s_1 $ and $t_k \to \infty$ as $t_{k+1} \geq t_k + b_*$, 
we know that $R_{\lambda_0}$ is covered by   $R_{\lambda_0}\subset \cup_{k=1}^N R_k$ 
by finitely many sets $R_k = R(t_k, m_*b)$
(where we let $N$ be the minimal such integer).
Thus, there exists  $1 \leq k \leq N$ such that $x_0  \in\bar \psi( R_k, s_{\lambda_0})$. 
Note that from the definition of $R_k$ and the minimality of $N$ we have $s_{\lambda_0}> t_k - m_* b  $.
Thus, \eqref{Strategy1} and the induced strategy \eqref{InducedStrategy} imply that
\be
\nonumber
	x_0 \in \psi(\omega_{k+1}) \subset \psi(\omega_k ) - \bar \psi(R_k, t_k +l_* m_* b) ,
\ee
and in particular,
\bea
\label{NotIn}
	x_0 &\notin&\bar \psi(R_k, t_k +l_* m_* b) 
				\\ \nonumber	
				&=& \bar \psi (R_k, t_k - m_* b  +(l_*+1)m_* b)
				 \supset \bar \psi (R_k, s_{\lambda_0} +(l_*+1)m_* b),
\eea
by \eqref{Mono}.
This shows that 
\be
\nonumber
	x_0 \notin \cup_{k=1}^N \bar \psi (R_k, s_{\lambda_0} + (l_*+1)m_* b  ) \supset \bar \psi(R_{\lambda_0}, s_{\lambda_0} + (l_*+1)m_* b).
\ee
Therefore, $x_0\in $  \textbf{Bad}$(\cal{F})$, since
 \be
 \nonumber
 	x_0 \notin \bigcup_{\lambda \in \Lambda}
	 \bar \psi(R_{\lambda}, s_{\lambda} + (l_*+1)m_* b).
\ee
Hence, $A$ wins and we defined a winning strategy for the parameter $b>b_*$.

Now, assume that $(b_*)$ is satisfied and note that in particular  \eqref{LegalPlay3} is satisfied with respect to the sets $A_k$ in \eqref{StrongStrategy1}.
Hence, since $(b_*)$ implies $(b_*, 1,0)$, the first part of the theorem and Lemma \ref{AbsWinWin} finish the proof.
\end{proof}

We want to show that the conditions are  preserved under maps which satisfy some kind of bi-Lipschitz-property and by finite intersections.

First, let $(\bar X, \bar \Omega_{\bar X},  \psi_{\bar X})$ and $(\bar Y,  \bar \Omega_{\bar Y}, \psi_{\bar Y})$ be two parameter spaces with monotonic functions.
For a given constant $L_* \geq 0$, 
consider a map $F : \bar X \to \bar Y$ such that 
\be
\label{Lipschitz}
	\psi_{\bar Y} (F(x), r+2L_*) \subset F(  \psi_{\bar X}(x,r+L_*)) \subset \psi_{\bar Y}(F(x), r),
\ee
for all formal balls $(x, r) \in  \Omega_{\bar X}$.
If both $ \psi_{\bar X}=B_1^{\bar X}$ and $\psi_{\bar Y}=B_1^{\bar Y}$, then $F$ is a $L_*$-bi-Lipschitz map.
Given a nested, discrete family of resonant sets $\cal{F}_{\bar X}=(\Lambda, R_{\lambda}, s_{\lambda})$,
consider the induced nested and discrete family in $\bar Y$,
\be
\nonumber
	\cal{F}_{Y} \equiv F(\cal{F}_{X}) \equiv (\Lambda, F(R_{\lambda}), s_{\lambda} - L_* ).
\ee
If $F$ is bijective, $X \subset \bar X$, then
it is readily checked that $F(\textbf{Bad}_X^{\psi_{\bar X}}(\cal{F}_X)) = \textbf{Bad}_{F(X)}^{\psi_{\bar Y}}(\cal{F}_{\bar Y} )$.

\begin{proposition}
\label{Preserving}
Let $(\bar X, \bar \Omega_{\bar X},  \psi_{\bar X})$, $(\bar Y,  \bar \Omega_{\bar Y},\psi_{\bar Y})$ 
and let $F : \bar X \to \bar Y$ be a bijective map which satisfies \eqref{Lipschitz}.
If $(\Omega_X, \psi_X)$ is [strongly $b_*$-diffuse] $(b_*,n_*, 2L_*)$-diffuse with respect to $\cal{F}_X$,
then, for $Y\equiv F(X)$, $(\Omega_Y, \psi_Y)$ is [strongly $(b_*+2L_*)$-diffuse] $(b_* + 2L_*, n_*, 0)$-diffuse with respect to $\cal{F}_Y$. 
\end{proposition}

\begin{proof}
Assume that $(\Omega_X, \psi_X)$ is $(b_*,n_*, 2L_*)$-diffuse with respect to $\cal{F}_X$.
Let $(y, r) \in \Omega_Y$ and $b>b_*$.
There exists $n\in \N$ and $\omega'=(\bar x, r + L_* +b)\in \Omega_X$  such that
\be
\label{Gleichung}
	\psi_X(\omega') \subset \psi_X(F^{-1}(y), r+L_*) - \psi_{\bar X}(R_{\bar X}(r+L_*, n_*(b + 2L_*)), r+L_* +nb ).
\ee
From \eqref{Lipschitz} we have 
\bea
\nonumber
	\psi_Y(F(\bar x), r +  (b +2L_*) ) &\subset& F(\psi_X(\bar x, r + L_*+b)) 
		\\ \nonumber
		&=& F(\psi(\omega'))
		\\ \nonumber
		&\subset& F(\psi_X(F^{-1}(y), r + L_*) ) \subset \psi_Y(y, r).
\eea
Note that  $F(R_X(r+L_*, t)) = R_Y(r, t)$.
We obtain
\bea
\nonumber
	 \psi_{\bar Y}(R_{\bar Y}(r, n_*(b+2L_*)), r+ n (b + 2L_*)) &\subset & 
		\psi_{\bar Y}(R_{\bar Y}(r,n_*(b +2L_*)), r+2L_* + nb) 
		 \\\nonumber
		&\subset & F( \psi_{\bar X}(R_{\bar X}(r+L_*, n_*(b+2L_*)), r+L_*+nb).
\eea
By \eqref{Gleichung} we know that $F(\psi_X(\omega'))$ is disjoint to
$F( \psi_{\bar X}(R(r+L_*, n_*(b+2L_*)), r+L_*+nb)$ and hence
we see that $(\Omega_Y, \psi_Y)$ is $(b_* + 2L_* ,n_*,0)$-diffuse with respect to $\cal{F}_Y$.

The case when $(\Omega_X, \psi_X)$ is strongly $b_*$-diffuse with respect to $\cal{F}_X$ follows similarly.
\end{proof}

Now consider finitely many families $\cal{F}_i = (\Lambda^i, R^i_{\lambda^i}, s^i_{\lambda^i})$, $i=1, \dots, n_*$,
 of nested and discrete families in $\bar X$.
When $(\Omega, \psi)$ is strongly $b_*$-diffuse with respect to each $\cal{F}_i$,
we know from Theorem \ref{Winning} and properties of $\psi$-modified Schmidt games
that $\cap_{i=1}^{n_*} \textbf{Bad}(\cal{F}_i)$ is $(\psi, b_*)$-winning (and the same is true for countable intersections).
In the weaker setting, we show the following.

\begin{proposition}
\label{Intersection}
If $(\Omega, \psi)$ is $(b_*,  n_*, L_*)$-diffuse with respect to each family $\cal{F}_i$,
then $\cap_{i=1}^{n_*} \textbf{Bad}(\cal{F}_i)$ is weakly $(\psi, b_*)$-winning.
\end{proposition}

\begin{proof}
Assume that $X$ is $(b_*,n_*,L_*)$-diffuse with respect to each family $\cal{F}_i$
and let $b>b_*$.
We only need to modify the strategy for player $A$ with respect to the sets $A_k$ in  \eqref{Strategy2}.
In fact, if $\omega_1=(x_1, t_1)\in \Omega$ is the first move of $B$, we let  again $m_* \in \N$ such that $\tilde b = m_*b\geq t_1 -s_1 $. 
Let $k= l n_* + s$ for $l\in \N_0$ and $1 \leq s \leq n_*$.
Denote by $R^s_l = R^{s}(t_k , n_* \tilde b)$, where $R^{s}$ is the subset of the resonant sets with respect to $\cal{F}_{s}$.
Moreover, let $l_*=l(\tilde b) = l_1 + \dots + l_{n_*}$, where $l_i= n_i(\tilde b)$ is the constant in \eqref{BarBDiffuse} with respect the family $\cal{F}_i$.
We therefore define
\be
\nonumber
	A_k = \bar \psi(R^s_l, t_k  +l_* \tilde b) \cap X.
\ee
By \eqref{BarBDiffuse}, there exists a formal ball
$\omega_{k+1} = (x_{k+1}, t_{k} + \tilde b) \in \Omega$ such that
\be
\nonumber
	\psi(\omega_{k+1}) \subset \psi(\omega_k) - \bar \psi(R^s_l, t_k  + l_* \tilde b) = \psi(\omega_k) - A_k,
\ee
which shows that the set $\cal{L}^{\psi}_{b}(\omega_{k})$ in \eqref{InducedStrategy} modified with respect to the set $A_k$ is nonempty.

Thus, for $s=1, \dots, n_*$ and  $x_0 \in \cap_{k \geq 1} \psi(\omega_k) $,
we deduce similarly to \eqref{NotIn} that  $x_0  \in \textbf{Bad}(\cal{F}_{s})$.
In particular, $x_0 \in \cap_s \ \textbf{Bad}(\cal{F}_{s})$ which is thus a $(\psi, b_*)$-weakly-winning set.
\end{proof}

Given $\bar Y_i, Y_i, \bar \psi_i$, $i=1, \dots, n_*$,  assume that $\cal{F}_i = (\Lambda_i, R^i_{\lambda^i}, s_{\lambda^i})$ 
is a nested discrete family in $\bar Y_i$ and that  $F_i : \bar Y_i \to \bar X$ is a bijective map satisfying \eqref{Lipschitz} for a constant $L_*$ with $F(Y_i)=X$.
As a corollary, if each $(\Omega_i, \psi_i)$ is $(b_*, n_*, L_*)$-diffuse with respect to $\cal{F}_i$, 
then
\be
\label{WeaklyIncompressible}
	\cap_{i=1}^{n_*} F_i(\textbf{Bad}_{Y_i}^{\bar \psi_i}(\cal{F}_i)) \subset X
\ee
is a weakly $(\psi_X, b_* + 2L_*)$-winning set. 

\begin{remark}
Let $\bar \Omega_i= \bar X_i \times (t_*,\infty)$ and $\bar \psi_i$ be given for $i=1,2$, 
where $\bar \psi_1 \times \bar \psi_2( x_1, x_2, t)= \bar \psi_1(x_1, t) \times \bar \psi_2(x_2, t)$.
Moreover, let $\cal{F}_i = (\Lambda, R^i_{\lambda}, s_{\lambda})$ be nested and discrete
with the same index set and the same sizes.
If $(b_*)$ or $(b_*,n_*, L_*)$  respectively is satisfied for both $X_i \subset \bar X_i$ and $\cal{F}_i$,
then $(b_*)$ or $(b_*,n_*, L_*)$  respectively is satisfied for $X_1\times X_2$
with respect to 
$\cal{F}=(\Lambda, R_{\lambda}^1 \times R_{\lambda}^2, s_{\lambda})$ and  $\bar \psi_1\times \bar  \psi_2$.
\end{remark}


\subsection{Diffuse spaces and absolutely decaying measures.}
\label{DiffuseSpaces}

In this subsection we first discuss diffusion properties of the subspace $X$ in $\bar X$, or rather of the parameter spaces $(\Omega, \psi)$ in $(\bar \Omega, \bar \psi)$,
and then relate these properties to the (local) structure and distribution of the resonant sets of a given family $\cal{F}$ in $\bar X$.

In the following, let  $X$ be a nonempty closed subset of a proper metric space $\bar X$ with a given monotonic function $\bar \psi$.
We give a special class of diffuse spaces $X $ in which the resonant sets might be more general than points 
but are still nicely structured and distributed.
More precisely, let $\cal{S} = \{S \subset \bar X\}$ be a given nonempty collection of subsets of $\bar X$. 
For instance, let $\cal{S}$ be the set of metric spheres $S(\bar \omega) \equiv \{y \in \bar X : d(\bar x,y)=e^{-t}\}$, where $\bar \omega=(\bar x, t) \in \bar \Omega$,
or the set of affine hyperplanes in $\R^n$.

For $b_*> 0$, $(\Omega,  \psi)$ is called \emph{$b_*$-diffuse with respect to $\cal{S}$},
if for any formal ball $\omega=(x,t) \in \Omega$ 
and any set $S\in \cal{S}$ 
there exists a formal ball $\omega' = (x', t+b_*) \in \Omega$ such that 
\be
\label{NbDiffuse}
	\psi(\omega') \subset \psi(\omega) - \bar \psi(S,  t + b_*).
\ee

For the standard function $\psi=B_1$, our definition  above is similar to the following special cases.
\begin{itemize}
\item[1.]
When $\bar X=\R^n$ is the Euclidean space and  $\cal{S}$ is the set of $k$-dimensional affine hyperplanes in $\R^n$ ($0\leq k <n$),
then $X \subset \R^n$ is called \emph{$k$-dimensionally hyperplane diffuse}; see \cite{BroderickEtAl}.

\item[2.] When $k=0$, that is, $\cal{S}$ is the set of points in a metric space $\bar X$, and $\beta=b_*$, then
$X \subset \bar X$ is called  \emph{$\beta$-diffuse}; see  \cite{MayedaMerrill}.
\end{itemize}

For a class of $\beta$-diffuse spaces, 
let $X$ be a \emph{uniformly perfect} metric space, that is, there exists $r_*\in \R \cup \{- \infty\}$ and a
constant  $0< \nu <\infty $ such that for any metric ball $B(x,e^{-r})$, $x \in X$, $r>r_*$ with $X-B(x,e^{-r}) \neq \emptyset$,
we have 
\be
\nonumber
	(B(x, e^{-r})-B(x,e^{-(\nu + r)}))\cap X \neq \emptyset.
\ee 
Similar to \cite{MayedaMerill}, Lemma 2.4, we show the following.

\begin{lemma} 
\label{UniformlyPerfect}
If $X$ is uniformly perfect with respect to $ \nu> 0$,
then $X$ is $\beta$-diffuse for any $\beta \geq  \nu + \log(4) + \log(4/3)$.
\end{lemma}

\begin{proof}
Let $x \in X$, $r>r_*$  and $\bar x \in \bar X$.
If $d(x,\bar x) > 2 e^{-(r + \beta)}$ then for $x'=x$ we have $B(x', e^{-(r + \beta)}) \subset B(x, e^{-r}) - B(\bar x, e^{-(r + \beta)})$.
On the other hand, if $d(x,\bar x) \leq 2 e^{-(r + \beta)}$ then $B(\bar x, e^{-(r + \beta)}) \subset B(x, 3e^{-(r + \beta)})$.
Let $c= \beta - \nu - \log(4) \geq \log(3/4)$.
Since $X$ is uniformly perfect, there exists $x' \in (B( x, e^{-(r  + c)})-B(x,e^{-(\nu + r  + c)})) \cap X$.
Hence,  
\be
\nonumber
	4e^{-(r+ \beta)}  \leq e^{-(r+\nu + c)} < d(x,x') \leq  e^{-(r +c )} \leq \tfrac{3}{4}e^{-r} \leq e^{-r} - e^{-(r + \beta )}.
\ee
Again we have $B(x', e^{-(\beta + r)} ) \subset B(x, e^{-r} ) - B(\bar x,e^{-(\beta + r)} )$.
\end{proof}

\noindent 
Consider the following examples of $b_*$-diffuse  spaces $X \subset \bar X$.
\begin{itemize}
\item[1.] 
If  $\Gamma$ is a non-elementary finitely generated Kleinian group acting on the hyperbolic space $\H^{n+1}$ (the unit ball model), 
then the limit set $X=\Lambda\Gamma \subset S^n=\bar X$ of $\Gamma$ is uniformly perfect by \cite{JarviVuorinen}. 
For the definitions see Subsection \ref{CAT(-1)}.
\item[2.] Let $n\geq 1$. If $\Sigma^+ = \{0,\dots,n\}^{\N}$ 
denotes the set of one-sides sequences in the symbols $\{0,1,\dots ,n\}$, 
together with the metric $d^+(w, \bar w)\equiv e^{-\min\{ i \geq 1: w(i) \neq \bar w(i) \}}$ for $w \neq \bar w$ and $d(w, w)\equiv0$, 
then $(\Sigma^+,d)$ is compact and $\beta$-diffuse for $\beta= 1$.
\item[3.]
Let $T$ be a tree of valence at least $3$ with the path metric such that every edge is of length $1$.
For a vertex point $o\in T$,
let $d_o$ be the visual metric (see Section \ref{CAT(-1)} for the definition) on the set $\partial T$ of ends of $T$.
Then $(\partial T, d_o)$ is compact and $1$-diffuse.
\item[4.]
If $X$ is the support of a locally finite Borel measure on $\bar X=\R^n$ which is absolutely $\delta$-decaying,
then there exists $b_* = b_*(\delta)>0$ such that $X$ is $(n-1)$-dimensionally $b_*$-diffuse.
For the definition and the proof see below.
Moreover,  the following result is due to \cite{LindenstraussEtAl}.
Let $\{S_1, \dots, S_k\}$ be an irreducible family of contracting self-similarity maps of $\R^n$ satisfying the
open set condition and let $X$ be the attractor.
If $\mu$ is the restriction of the $\delta$-dimensional Hausdorff-measure to $X$, $\delta=$ dim$(X)$,
then $\mu$ is absolutely $\a$-decaying and satisfies a power law with respect to the exponent $\delta$.
Particular examples of such sets are regular Cantor-sets, Koch's curve and the Sierpinski gasket.
\end{itemize}

In the following, consider a nested and discrete family $\cal{F} = ( \Lambda, R_{\lambda}, s_{\lambda})$  of resonant sets in $\bar X$.
We are interested in properties of $\cal{F}$ such that condition $(b_*)$ is 'inherited' from a given structure of the parameter space.
The family $\cal{F}$ is called \emph{locally contained in $\cal{S}$} (with respect to $(\bar \Omega, \bar \psi)$)
if there exists $ l_*\geq 0$
and a number $n_*\in \N$
such that for all $(x,t) \in \Omega$ we have
\be
\label{ContainedInSpheres}
	\bar \psi(x, t+ l_* ) \cap R(t) \subset \bigcup_{i=1}^{n_*} S_i
\ee
is contained in at most $n_*$  sets $S_i$ of $\cal{S}$.%
\footnote{ Note that if $\bar \psi(x, t+ l_* ) \cap R(t)$ is empty, \eqref{ContainedInSpheres} is trivially satisfied. }

For a constant $d_*>0$, we say that the parameter space $(\Omega, \psi)$ is \emph{$d_*$-separating} if for all formal balls $(x,t) \in \Omega$
and for any set $M$ disjoint to $\bar \psi(x,t)$, we have
\be
\label{Separating}
		\bar \psi(x,t + d_*) \cap \bar \psi(M, t + d_*) = \emptyset.
\ee
Clearly, the standard function $B_{\sigma}$ is $\log(3)/\sigma$-separating in a proper metric space $\bar X$.

\begin{theorem} 
\label{DiffuseProposition}
Let $(\Omega, \psi)$ be $b_*$-diffuse with respect to $\cal{S}$,  $d_*$-separating and $\cal{F}$ be locally contained in $\cal{S}$ with $n_*=1$.

Then $(\Omega, \psi)$ is strongly $\bar b_*$-diffuse with respect to $\cal{F}$ where $\bar b_*=l_*+ d_* + b_*$.
Hence, $\textbf{Bad}(\cal{F})$ is $(\psi, \bar b_*)$-winning and moreover absolute  $(\psi, \bar b_*)$-winning with respect to $\cal{S}$.%
\footnote{ We remark that in \eqref{HAW} we considered a collection $\cal{S}$ of sets in $X$ instead of $\bar X$ since the supspace $\bar X$ was not yet introduced.  }
\end{theorem}

\noindent 
Note that Theorem \ref{DiffuseProposition} with $\psi=B_{\sigma}$ implies Theorem \ref{IntroductionThm}.

\begin{proof}
Given $(x,t ) \in \Omega$ and $l_*$, as well as $S\in \cal{S}$ from the definition of  \eqref{ContainedInSpheres},
we claim that, 
for $s\geq 0$, 
\be
\label{ContainedInS}
	\psi(x, t + l_* + d_*) \cap \bar \psi( R(t), t  + l_* +  d_* + s) \subset \psi( x, t+ l_* + d_*) \cap \bar \psi(S, t  + l_* + d_* + s).
\ee
In fact, let $M$ be the set  $R(t) - S$ which is disjoint to $\bar \psi(x, t + l_*)$ by   \eqref{ContainedInSpheres}.
The $\bar \psi$-ball $\bar \psi(x, t+l_* + d_*)$
is, by  \eqref{Separating}, disjoint to 
\be
\nonumber
	\bar \psi(M, t + l_* + d_*) \supset \bar \psi(M, t + l_* + d_* + s),
\ee
for $s\geq 0$.
This shows the above claim. 

Set $\bar b_* = l_* +d_* +b_*$.
Since $(\Omega, \psi)$ is $b_*$-diffuse with respect to $\cal{S}$, applied to the formal ball $\omega=(x,t + l_* + d_*)$,
there exists $\omega' =(x', t+ l_* + d_* + b_*) = (x', t +\bar b_*) \in \Omega$ as in \eqref{NbDiffuse}.
In particular, by monotonicy $\bar \psi$, we obtain for every $\lambda\in \Lambda$ with $s_{\lambda} \leq t$ that
\bea
\label{ChooseSpheres}
	\psi(\omega') &\subset& \psi(x, t+ l_* + d_*) - \bar \psi( S, t + l_* + d_* +b_*) 
	\\ \nonumber
	&\subset& \psi(x, t+ l_* + d_*) - \bar \psi( R(t), t + l_* + d_* +b_*) 
	\subset \psi(x, t) - \bar \psi( R_{\lambda}, t + \bar b_*).
\eea
This shows that $(\Omega, \psi)$ is strongly $\bar b_*$-diffuse with respect to $\cal{F}$.

In fact, \eqref{ChooseSpheres} shows that we can even choose, for a parameter $b\geq \bar b_*$, 
\be
\nonumber
	 \bar \psi( S, t_k +\bar b_*) \supset A_k \equiv \bar \psi( S, t_k +b)  \cap X \supset \bar \psi( R(t_k), t_k + b) \cap X
\ee
in \eqref{StrongStrategy1} and \eqref{StrongInducedStrategy} respectively.
Following the proof of Theorem \ref{Winning} shows that $\textbf{Bad}(\cal{F})$ is an absolute  $(\psi, \bar b_*)$-winning set with respect to $\cal{S}$ (as define in \eqref{HAW}).
\end{proof}

As a special case, let $\bar \psi =  B_{\sigma}$ be the standard function and $\bar X$ be a proper metric space.
Recall that $d_* \leq \log(3)/\sigma$,
and assume that for all distinct points $x$, $ y \in R_{\lambda}$ we have
\be
\label{Distinct}
	d(x,y) > \bar c \cdot e^{- \sigma s_{\lambda}},
\ee
for some constant $\bar c>0$.
It is readily checked that, setting $l_* = -\log(\bar c) + \log(2)$ and $\cal{S}$ to be the set of points,
the following is a corollary of Theorem \ref{DiffuseProposition} with $\sigma=1$.

\begin{proposition} 
\label{Sufficient}
Assume that $X$ is $\beta$-diffuse.
If \eqref{Distinct} is satisfied for $\sigma=1$, 
then $(\Omega, B_{1})$ is strongly $\bar b_*$-diffuse with respect to $\cal{F}$, where $\bar b_*=  -\log(\bar c)  + \log(2)  + d_*+ \beta$.
In particular, \textbf{Bad}$(\cal{F})$ is absolute-winning (in the sense of McMullen).
\end{proposition}

\begin{remark}
Note that Condition \eqref{Distinct} (with $\sigma=1$) is similar to, but in fact weaker than the condition
\be
\nonumber
	d(x, y) \geq \sqrt{e^{-s_{\lambda}} e^{- s_{\lambda'}}},
\ee
for $x \in R_{\lambda}$, $y\in R_{\lambda'}$.
For $X=\R^n$, this condition was considered in a similar setting by \cite{Dani2} and 
recently by \cite{DaniShah} where it was called $\cal{B}$-set.
\end{remark}

For another class of examples of diffuse spaces, we extend the notion of absolutely decaying measures on $X$, introduced in \cite{LindenstraussEtAl},  
to the setting of parameter spaces $(\Omega, \psi)$ and collections $\cal{S}$.
Note that in the Euclidean setting, already \cite{Fishman} and \cite{BroderickEtAl} used absolutely decaying measures in relation with Schmidt games.

A subset $S\subset \bar X$ is called \emph{$\bar \psi$-Borel}, if $\bar \psi(S, t)$ is a Borel set for all $t>t_*$.
Assume that every Borel set in $X$ is $\bar \psi$-Borel.

Given a locally finite Borel measure $\mu$ with supp$(\mu)=X$ and a collection $\cal{S} \equiv \{ S \subset \bar X \}$  of $\bar \psi$-Borel sets, 
$(\Omega, \psi, \mu)$ is said to be \emph{absolutely $(\delta, c_{\delta} )$-decaying with respect to $\cal{S}$},
where $\delta$, $c_{\delta}>0$,
if for all $(x,t) \in \Omega$ and $S \in \cal{S}$ we have for all $s\geq 0$ that
\be
\label{SphereDecaying}
	\mu(\psi(x,t) \cap \bar \psi(S, t + s)) \leq c_{\delta} e^{- \delta s} \ \mu(\psi(x,t)).
\ee
The function $f(s) =  c_{\delta} e^{- \delta s}$ determines the rate of the decay of  the measure of 
 $\bar \psi(S, t + c)$ in $\psi(x,t)$ in terms of the relative size $s$ of the $\bar \psi$-neighborhood of $S$.

Clearly, if $\psi=B_1$ and $\cal{S}$ denotes the collection of affine hyperplanes in $\bar X=\R^n$,
$\mu$ corresponds to an \emph{absolutely $\delta$-decaying} measure in the classical sense (see \cite{LindenstraussEtAl}).

We say that $(\Omega, \psi)$ is \emph{$d_*$-separating with respect to $\cal{S}$}, if
for all formal neighborhoods $(S,t) \in \cal{P}$, $S\in \cal{S}$ and all $x, y \in X$,
\bea
\label{SeparatingSpheres}
	x\not \in \bar \psi(S,t) &\implies& \psi(x,t + d_*) \cap \bar  \psi(Y, t + d_*) = \emptyset 
	\\ \nonumber
	x\in \psi(y,t+ d_*) &\implies& \psi(x,t+ d_*) \subset  \psi(y,t).
\eea
Clearly, if  $\bar X=\R^n$ and $S$ is an affine hyperplane, then $B_{\sigma}$ is $\log(2)/\sigma$-separating with respect to $\cal{S}$.

\begin{proposition}
\label{AbsDecMeasure}
Let $(\Omega, \psi, \mu)$ be absolutely $(\delta, c_{\delta} )$-decaying with respect to $\cal{S}$
and $(\Omega, \psi)$ be $d_*$-separating with respect to $\cal{S}$.
Then $(\Omega, \psi)$ is $b_*$-diffuse with respect to $\cal{S}$ for all $b_* >  \log(c_{\delta})/\delta + 2d_*$.
\end{proposition}

\begin{proof}
We only sketch the proof since it is very similar to the proof of Proposition \ref{DecayingMeasure} below.
Given a formal ball $\omega=(x,t) \in \Omega$ and $S \in \cal{S}$,
condition \eqref{SphereDecaying} applied to $\omega'=(x, t+ 2d_*)$
implies the existence of a point $x' \in \psi(\omega')  - \bar \psi(S, t + d_* + s)$, for all $s\geq s_0>  \log(c_{\delta})/\delta$.
Hence, \eqref{SeparatingSpheres} shows for the formal ball $\bar \omega=(x', t + 2d_*+s_0) \in \Omega$
that $\psi(\bar \omega)$ is contained in $\psi(\omega)$ and disjoint to $\bar \psi(S, t + 2d_* + s_0)$.
\end{proof}

As a further tool to show that a parameter space satisfies $(b_*)$ or $(b_*, n_*, L_*)$  with respect to a given family $\cal{F}$, 
we want to extend the notion of absolutely decaying measures. 
Let $X$ be the support of a locally finite Borel measure $\mu$. 
Moreover,  let $f : [0, \infty] \times \R \to [0, \infty)$ be a function, non-decreasing in the first and non-increasing in the second argument,
where we denote $f_b(\cdot) \equiv f(b, \cdot)$.
If every resonant set $R_{\lambda}$ is $\bar \psi$-Borel,%
\footnote{ In this case, also $R(r,b) = R(r ) - R(r-b)$ is $\bar \psi$-Borel for every $r\in \R$, $b>0$.}
 we call the family $\cal{F}$
 \emph{measurable} and, for a function $f$ as above, consider the following conditions.

\begin{itemize}
\item[$(\mu^s)$] $(\Omega, \psi, \mu)$ is called  \emph{strongly (absolutely) $f$-decaying with respect to $\cal{F}$},
	if  for all formal balls $\omega=(x,r) \in \Omega$  and for all $s\in \R$ we have
\be
\nonumber
	\mu(\psi(\omega ) \cap  \bar \psi(R(r ), r+s) ) \leq f_{\infty}(s) \ \mu(\psi(\omega)).
\ee
\item[$(\mu)$]  $(\Omega, \psi, \mu)$ is called  \emph{(absolutely) $f$-decaying with respect to $\cal{F}$},
	if  for all formal balls $\omega=(x,r) \in \Omega$, for all $b\geq 0$ and $s\in \R$ we have
\be
\label{DecayingMeasure}
	\mu( \psi(\omega) \cap \bar \psi(R(r,b), r+s ) ) \leq f_b(s) \ \mu(\psi(\omega)).
\ee
\end{itemize}
Again, the function $f$ determines the rate of decay of the measure of the relative neighborhood of the resonant set in $\psi(x,r)$.
For constants $n_*\in \N$, and $d_*, L_* \geq 0$, $b_*> 2d_*$, we say that $f$ is \emph{$(d_*, b_*,n_*,L_*)$-decaying}
if there exists a constant $c_0<1$  such that
\be
\label{Decaying}
	f( n_*(b+L_*) + d_*, b - 2d_*) \leq c_0 \ \ \  \text{ for all $b> b_*$,}
\ee
and strongly \emph{$(d_*,b_*)$-decaying} if $f_{\infty} (b_* - 2d_*)\leq c_0$.

\begin{proposition} 
\label{SphereDecayingMeasure}
Let $(\Omega, \psi, \mu)$ be absolutely $(\delta, c_{\delta})$-decaying with respect to $\cal{S}$, 
and $(\Omega,  \psi)$ be $d_*$-separating.
If moreover $\cal{F} = (\Lambda, R_{\lambda}, s_{\lambda})$ is locally contained in $\cal{S}$ for $n_*\in \N$ and $l_*\geq 0$, 
then $(\Omega, \psi, \mu)$ is (absolutely) $f$-decaying with respect to $\cal{F}^* = (\Lambda, R_{\lambda}, s_{\lambda} +  l_* + d_*)$
where $f_{\infty}(s) =  n_* c_{\delta} e^{- \delta s} $ is $(d_*, b_*)$-decaying for $b_*> \log(n_* c_{\delta} e^{-2\delta d_*})/\delta$.
\end{proposition}

\begin{proof}
Using the argument of Claim \eqref{ContainedInS}, the proof is straight foreward and left to the reader.
\end{proof}

We say that the parameter space $(\Omega, \psi)$ is \emph{$d_*$-separating with respect to $\cal{F}$}, if  
there exists a constant $d_*>0$ such that for all formal neighborhoods $(Y,t) = (R(r, b),t) \in \cal{P}$, 
or formal balls $(Y,t) = (y,t) \in \Omega$ and for all $x \in X$,
\bea
\label{SeparatingF}
	x\not \in \bar \psi(Y,t) &\implies& \psi(x,t + d_*) \cap \bar  \psi(Y, t + d_*) = \emptyset.
	\\ \nonumber
	x\in \psi(y,t+ d_*) &\implies& \psi(x,t+ d_*) \subset  \psi(y,t).
\eea
Clearly, if  $\bar X$ is a proper metric space and every $R_{\lambda}$ is a discrete set,
then the standard function $B_{\sigma}$ is $\log(3)/\sigma$-separating with respect to $\cal{F}$.

\begin{proposition}
\label{DecayingMeasure}
Let $(\Omega, \psi)$ be $d_*$-separating with respect to $\cal{F}$ and  $\mu$ be a locally finite Borel measure  with $X= $ supp$(\mu)$.
If  $(\Omega, \psi, \mu)$  is [strongly] absolutely $f$-decaying with respect to $\cal{F}$ and a function $f$ which
is [$(d_*, b_*)$-decaying] $(d_*, b_*,n_*, L_*)$-decaying,
Then $(\Omega, \psi)$ is [strongly $\bar b_*$-diffuse] $(\bar b_*, n_*, L_*)$-diffuse with respect to $\cal{F}$, where $\bar b_* =  b_* + 2d_*$.
\end{proposition}

\begin{proof}
Assume that   $(\Omega, \psi, \mu)$ is $f$-decaying with respect to $\cal{F}$ and $f$ is $(d_*,b_*, n_*, L_*)$-decaying.
For $\bar b_* =  b_* + 2d_*$ and $b>\bar b_*$ 
note that $R(r,n_*b ) \subset R(r + d_*, n_*b + d_*)$ and $b - 2d_* \geq b_*$.
Let $\omega=(x,r) \in \Omega$ with $r>r_*$.
We have
\bea
\nonumber
	&&\mu( \psi(x, r + d_* ) \cap \bar \psi(R(r,n_*(b+L_*)), r+ b- d_*))
	\\ \nonumber
	&\leq & 
	\mu( \psi(x, r + d_* ) \cap \bar \psi(R(r + d_*, n_*(b+L_*) + d_* ), r + d_* + ( b  -2 d_*)))
	\\
	\nonumber
	&\leq& 
	 f( n_*(b+L_*)+ d_* , b- 2d_*)   \mu( \psi(x, r + d_*) ).
\eea
Since for $b> \bar b_*= b_* +2d_*$  we have $ f( n_*(b+L_*)+d_*, b- 2d_*)   \leq c_0 <1$ by \eqref{Decaying},
there exists a point $\bar x \in \psi(x, r+ d_*) \cap \bar \psi(R(r,n_*(b+L_*)), r+ b - d_*)^C$.
By \eqref{Separating} and  since $b>d_*$,  
we have for $\omega'=(\bar x, r+ b) \in \Omega$ that $\psi(\omega') \subset \psi(\bar x,  r + d_*) \subset \psi(x, r)$.
Furthermore, \eqref{SeparatingF} implies that $\psi(\bar x, r +b)$ is disjoint from $\bar  \psi(R(r,n_*(b+L_*)), r+ b)$.
This shows that $(\Omega, \psi)$ is $(\bar b_*,n_*, L_*)$-diffuse with respect to $\cal{F}$.

The case when  $(\Omega, \psi, \mu)$ is strongly $f$-decaying follows similarly.
\end{proof}

We say that, given a locally finite Borel measure $\mu$ on $X = $ supp$(\mu)$, $(\Omega, \psi, \mu)$ satisfies a \emph{power law},
if there are parameters $\tau$, $c_1$, $c_2>0$, 
such that for all $\omega= (x,t)\in \Omega$ we have
\be
\nonumber
	c_1 e^{- \tau t } \leq \mu(\psi(x,t)) \leq c_2 e^{- \tau t}.
\ee
Note that $\tau$ might differ from the lower pointwise dimension of $\mu$ at a point and 
that clearly $(\mu1)$ is satisfied.

\begin{theorem}
\label{Dimension}
Let $(\Omega,  \psi)$ be $d_*$-separating with respect to $\cal{F}$ and let $(\Omega, \psi, \mu)$ satisfy a power law.
Assume that either $(\Omega, \psi, \mu)$ is $f$-decaying with respect to $\cal{F}$ where $f$ is $(d_*,b_*, 1, 0)$-decaying
or that $(\Omega, \psi)$ is strongly $b_*$-diffuse with respect to $\cal{F}$.
If moreover (MSG1-2) are satisfied,
then for all nonempty open sets $ U \subset X$, we have
\be
\nonumber
	\text{dim} (\textbf{Bad}(\cal{F}) \cap U) \geq d_{\mu}(U).
\ee
\end{theorem}

\begin{proof}
Let first $(\Omega, \psi, \mu)$ is $f$-decaying  with respect to $\cal{F}$  where $f$ is $(d_*,b_*, 1, 0)$-decaying.
Let $b> \bar b_* = b_*+d_*$ and $\omega_1=(x_1, t_1) \in \Omega$ be the first move of $B$ such that, by (MSG1), 
$\psi(\omega_1) \subset U$.
Let again $m_* \in \N$ with $\tilde b= m_*b\geq t_1$. 
For $k  \geq 1$, 
let $\omega_k = (x_k, t_k)$ be a choice of $B$.
As in the proof of Proposition \ref{DecayingMeasure} (with $n_*=1$, $L_*=0$), 
let $x^1 \in  \psi(x_k, t_k+d_*) \cap\bar \psi(R(t_k , \tilde b ), t_k+  \tilde b - d_* )^C$.
We moreover see that
\bea
\nonumber
	&& \mu(  \psi(x_k, t _k+ d_* )\cap \big( \psi( x^1, t_k + \tilde b - d_*) \cup \bar  \psi(R(t_k+d_*, \tilde b + d_*), t+  \tilde b -  d_*) \big))
	\\ \nonumber
	& \leq& c_2  e^{-\tau(t_k+\tilde b- d_*)}  +  f(\tilde b + d_* , \tilde b - 2d_*) \cdot \mu( \psi(x_k, t_k +d_* )
	\\ \nonumber
	& \leq&  (\tfrac{c_2}{c_1} e^{2\tau d_*} e^{-\tau \tilde b} + c_0)  \mu( \psi(x_k, t_k + d_*) ).
\eea
Since $c_0<1$ and $ \mu( \psi(x_k, t_k + d_*) )>0$, for $\tilde b$ sufficiently large such that $\tfrac{c_2}{c_1}e^{2\tau d_*} e^{-\tau \tilde b} + c_0<1$, 
there exists a point 
\be
\nonumber
	x^2 \in  \psi(x_k, t_k+d_*) \cap\bar  \psi(R(t_k+d_*,\tilde b + d_*), t_k+ \tilde b -  d_*)^C \cap \psi(x^1, t_k+ \tilde b - d_*)^C.
\ee
With the same arguments as above, 
 $\psi(x^2, r+\tilde b)$ is contained in $\psi(x_k,t_k)$ and  disjoint from  both, 
$\psi(x^1, t_k+\tilde b)$ and $\bar \psi(R(t_k + d_*, \tilde b + d_*), t_k+ \tilde b)$.
Iterating this argument until 
\be
\nonumber
	(N+1) \tfrac{c_2}{c_1} e^{2\tau d_*} e^{-\tau \tilde b} + c_0>1 ,
\ee 
we obtain $N$ points $x^1, \dots, x^N$ 
such that $\psi(x^i, t_k+\tilde b) \subset \psi(x_k,t_k)$, $i=1, \dots, N$, are disjoint and also disjoint to $\bar \psi(R(t_k+d_*,\tilde b + d_*), t_k+\tilde b)$.
Moreover, we have
\bea
\nonumber
	\mu \big(\bigcup_{i=1}^N \psi(x^i, t_k+ \tilde b) \big)
	&\geq&  N c_1  e^{-\tau(t_k+ \tilde b)} 
	\\ \nonumber
	&\geq&  \tfrac{(N+1)c_1}{2} e^{-\tau(t_k+ \tilde b)} 
	\\ \nonumber
	&\geq& \tfrac{(1-c_0)c_1^2e^{-2\tau d_*}}{2c_2^2} \mu(\psi(x_k,t_k))
	\equiv \bar c_0 \cdot \mu(\psi(x_k,t_k)),
\eea
Furthermore, each of the formal balls $\omega^i=(x^i, t_{k+1})= (x^i, t_k+\tilde b)  \in \cal{L}_{b}^{\psi}(\omega_k)$
is a  legal move according to the $(\psi, \bar b_*, b)$-winning-strategy of $A$ defined in \eqref{InducedStrategy}. 
This shows ($\mu2$)  for the parameter $b$ with $c = c( b) \geq \bar c_0$ and $m_*$.
Finally, Proposition \ref{LowerDimension2} implies that
\be
\label{Dim}
	\text{dim}( \textbf{Bad}(\cal{F})\cap U) \geq d_{\mu}(U) + \frac{ \log(\bar c_0)}{\sigma m_* b},
\ee
and the proof follows since \eqref{Dim} is true for every $b> \bar b_*$.

If $(\Omega, \psi)$ is strongly $b_*$-diffuse with respect to $\cal{F}$, there exists $\psi(\bar x, t+b_*) \subset \psi(\omega) - \bar  \psi(R(t), t+nb_*)$.
With similar arguments, we can choose disjoint formal balls $\psi(x_i, t+b)$, $i=1, \dots, N$, contained in $\psi(\bar x, t+b_*)$,
where $N$ is such that $(N+1) \tfrac{c_2}{c_1} e^{2\tau d_*} e^{-\tau b}> 1 $ and each of the formal balls is a legal move  according the to the $(\psi, b_*, b)$-winning strategy of $A$.
The proof then follows similarly.
\end{proof}

\begin{remark}
If we modify the requirements and the proof of Theorem \ref{Dimension} with respect to finitely many families $\cal{F}_i$, $i=1, \dots, n_*$,
where in particular $(\Omega, \psi, \mu)$ is $f_i$-decaying with respect to $\cal{F}_i$, $f_i$ is $(d_*, b_*, n_*, 0)$-decaying,
then we can show the result for $\textbf{Bad}(\cal{F}) $ replaced by $\cap_{i=1}^{n_*} \textbf{Bad}(\cal{F}_i)$. 
Moreover, if actually $X=F_i(Z)$ for bijective maps $F_i : \bar Z \to \bar X$ satisfying \eqref{Lipschitz} 
and $\cal{F}_i = F_i( \cal{F}_Z^i)$ with families $ \cal{F}_Z^i$ in $\bar Z$,
we obtained that
\be
\nonumber
	\text{dim} (\cap_{i=1}^{n_*} F(\textbf{Bad}(\cal{F}_Z^i)) \cap U) \geq d_{\mu}(U),
\ee
for any nonempty open set $U \subset X$.
This is a weaker version of the property that winning sets for Schmidt's game are incompressible; compare with \eqref{WeaklyIncompressible}. 
\end{remark}


\section{Applications}
\label{Apps}

\noindent
In order to discuss our conditions, we consider several examples from metric number theory (Part I.) and from dynamical systems (Part II.).
Given a complete metric space $X$ in $\bar X$ with a monotonic function $\bar \psi$  on $\bar \Omega$,
we are left with defining a suitable nested discrete family of resonant sets $\cal{F}$,
verifying the Conditions $(b_*)$ or $(b_*,n_*, L_*)$ respectively as well as finding suitable measures for the purpose of determining the Hausdorff-dimension of \textbf{Bad}$_X^{\bar \psi}(\cal{F})$.
\\

\paragraph{I. Examples from Number Theory.}

\subsection{\textbf{Bad}$_{\R^n}^{\bar r}$}
\label{BadRN}
For $n\geq 1$, let $\bar r \in \R^n$ with $r^1,\dots, r^n \geq 0$ such that $\sum r^i =1$.
Let \textbf{Bad}$_{\R^n}^{\bar r}$ be the set of points $\bar x = (x_1,\dots ,x_n)\in \R^n$ 
for which there exists a positive constant $c(\bar x)>0$ such that
\be
\nonumber
\label{Bedinung}
 	\max_{i=1, \dots, n} \lvert q x_i - p_i \rvert^{1/r^i}  \geq c(\bar x)/q,
\ee
for every $q\in \N$ and $\bar p = (p_1,\dots, p_n) \in \Z^n$.
The set \textbf{Bad}$_{\R^n}^{n} \equiv \textbf{Bad}_{\R^n}^{(1/n,\dots ,1/n)}$ agrees with the set of badly approximable vectors.

Let $\bar \Omega= \R^n \times (0, \infty)$ and define the monotonic function $\bar \psi = \bar \psi_{\bar r}$ by
\be
\nonumber
	\bar \psi( \bar x, t  ) \equiv 
	B(x_1, e^{-(1+r^1)t} )\times \dots \times B(x_n, e^{-(1+r^n)t} ).
\ee
While \cite{BroderickEtAl} showed that \textbf{Bad}$_{\R^n}^{n} \cap X$, where $X$ is the support of an absolutely decaying measure, is hyperplane absolute winning,
\cite{KleinbockWeiss} showed that \textbf{Bad}$_{\R^n}^{\bar r}$ is a winning set for the $\psi$-modified game. 
We want to combine these results and improve them to the following, 
where we set $\textbf{Bad}_{X}^{\bar r} \equiv \textbf{Bad}_{\R^n}^{\bar r}\cap X$ and let $\cal{S}$ be the collection of affine hyperplanes in $\R^n$.

\begin{theorem}
\label{BadN}
Let $X$ be the support of a locally finite Borel measure $\mu$ such that $(\Omega, \psi, \mu)$  is absolutely $(\delta, c_{\delta})$-decaying with respect to  $\cal{S}$.
Then, $\textbf{Bad}_{X}^{\bar r}$ is absolute $\psi$-winning with respect to $\cal{S}$.
\end{theorem}

\noindent 
Before we proof the Theorem, let $\mu$ be the Lebesgue-measure on $\R^n$.
Note that $(\Omega, \psi, \mu)$ is absolutely $(\delta, c_{\delta})$-decaying with respect to $\cal{S}$, for $\delta = 1+ \min \{r^1,\dots, r^n\}$ and $c_{\delta}>0$.
Moreover, $(\Omega, \psi, \mu)$ satisfies a power law with respect to the exponent $n+1$;
in fact, for all $(x, t)\in \Omega$ we have
	 $\mu(\psi(x,t)) = 2^n e^{- (n+1)t}$.

More precisely, let $\mu_i$ on $X_i\subset \R$ such that $(\Omega_i, B_{\sigma_i}, \mu_i)$ satisfies a power law with respect to the exponent $\tau_i$, $i=1, \dots, n$.
In particular, $(\Omega_i, B_{\sigma_i}, \mu_i)$ is absolutely $\tau_i$-decaying and $( (\times_{i=1}^n X_i) \times \R^+, \times_{i=1}^n B_{\sigma_i},\times_{i=1}^n \mu_i)$ satisfies a power law with respect to the exponent $\tau = \sum_{i=1}^n \tau_i$.
Moreover, using the arguments of \cite{LindenstraussEtAl}, Lemma 9.1, the following Lemma can be shown. 

\begin{lemma}
\label{ProductDecaying}
Assume that $(X_i \times \R^+, \psi_i, \mu_i)$, $X_i \subset \R^{n_i}$, is absolutely $(\delta_i, c_{\delta_i)}$-decaying with respect to 
affine hyperplanes in $\R^{n_i}$, $i=1,2$.
Then $(X_1 \times X_2 \times \R^+, \psi_1 \times \psi_2, \mu_1 \times \mu_2)$ is absolutely $(\delta, c_{\delta})$-decaying with respect to $\cal{S}$
for $\delta = \min\{\delta_i\}$, $c_{\delta} =\max\{c_{\delta_i}\}$.
\end{lemma}

\begin{proof}[Sketch of the proof]
Given the box $\psi((x_0,y_0), t) = \psi_1(x_0,t) \times \psi_2(y_0,t)$ and an affine hyperplane $S$ in $\R^{n_1 +n_2}$, 
we may, up to interchanging the role of the indices, assume that each slice $S_x \equiv S \cap   \{x\} \times \R^{n_2}$ is an affine hyperplane in $\R^{n_2}$.
Hence, write  
\be
\nonumber
	\psi((x_0,y_0), t) \cap \bar \psi(S, t + s) = \bigcup_{x \in \psi_1(x_0,t)} \{x\} \times \big( \psi_2(y_0, t) \cap \bar \psi_2(S_x, t+ s)  \big).
\ee
Disintegrating into the slices parallel to $\R^{n_2}$ and using that $\mu_2$ is absolutely $(\delta_2, c_{\delta_2})$-decaying, we obtain
\bea
\nonumber
	\mu(\psi((x_0,y_0), t) \cap \bar \psi(S, t + s) )
		&\leq& \mu_1( \psi(x_0,t)) \cdot c_{\delta_2}e^{- \delta_2 s} \mu_2( \psi_2(y_0, t))
	\\ \nonumber
		&\leq& c_{\delta} e^{-\delta s} \mu_1( \psi(x_0,t))\mu_2( \psi(y_0,t)) = c_{\delta} e^{- \delta s}\mu(\psi((x_0,y_0), t) ),
\eea
showing the claim. 
\end{proof}

So let $X$ be a product space as above, and note  that conditions (MSG1-2) are satisfied.  
By Theorem \ref{Dimension} (which we will see is applicable),
for any nonempty open set $U \subset X$, we have 
\be
\nonumber
	 \text{dim}( \textbf{Bad}_{X}^{\bar r} \cap U) \geq d_{\mu}(U);
\ee
this strengthens \cite{KristensenEtAl}, Theorem 11.

\begin{proof}[Proof of Theorem \ref{BadN}]
For $k\in \Lambda \equiv \N_{\geq 2}$ we define the set of rational vectors 
\be
\nonumber
	R_{k} \equiv \{ \bar p/q : \bar p \in \Z^n, 0<q<k \}
\ee
as resonant set and define its size by $s_k \equiv \log(k)$. 
The family $\cal{F}=( \N_{\geq 2}, R_k, s_k)$ is nested and discrete
and we want to show that $(\Omega, \psi)$ is strongly $b_*$-diffuse with respect to $\cal{F}$.

We use the following version of the 'Simplex Lemma' due to Davenport and Schmidt.

\begin{lemma}[\cite{KristensenEtAl}, Lemma 4]
\label{SimplexLemma}
Let $D\subset \R^n$ be a box of Euclidean volume  vol$(D) < 1/(n! k^{n+1} )$.
Then, there exists an affine hyperplane $L$ such that $R_k \cap D \subset L$.
\end{lemma}

\noindent Assuming the lemma for the moment, choose any resonant set $R_k$
and let $\omega= (x, r) \in \Omega$ be a formal ball such that $s_k\leq r$.
Note that, for $l_*> \log(n!\cdot 2^n)$, $\bar \psi(x, r+ l_*)$ is a box of Euclidean volume
\be
\nonumber
	 2e^{-(1+r^1)(r + l_*)} \cdots 2e^{-(1+r^n)(r + l_*)}= 2^n e^{-(1+n) (r+l_*)} <  \tfrac{1}{ n! k^{n+1}}.
\ee
The Simplex Lemma implies that $\bar \psi(x, r + l_*) \cap R_k \subset  L$, where $L \in \cal{S}$,
which shows that $\cal{F} $ is locally contained in $\cal{S}$ for $n_*=1$.

It is readily checked that $(\Omega, \psi)$ is $d_*$-separating  as well as $d_*$-separating with respect to $\cal{S}$, for $d_*=\log(3)/(1+\min\{r^i\})$.
Since $(\Omega, \psi, \mu)$ is $(\delta, c_{\delta})$-decaying, Proposition \ref{AbsDecMeasure} implies that $(\Omega, \psi)$ is $b_*$-diffuse with respect to $\cal{S}$, for $b_* > 2d_* + \log(c_{\delta}) /\delta$.
Thus, Theorem \ref{DiffuseProposition} shows that $(\Omega, \psi)$ is strongly $\bar b_*$-diffuse with respect to $\cal{F}$ where $\bar b_*=l_*+ d_* + b_*$
and that $\textbf{Bad}(\cal{F})$ is an absolute  $(\psi, \bar b_*)$-winning set with respect to $\cal{S}$.%

Finally, if $\bar x \in $ \textbf{Bad}$_X^{\psi}(\cal{F})$,
there exists a constant $c= c(\bar x)< \infty$ such that for all $\bar p/q$, where $\bar p=(p_1,\dots ,p_n) \in \Z^n$ and  $q\in \N$,
\be
\nonumber
	\bar x \not \in \bar \psi(R_{q+1}, s_{q+1} + c) \supset \bar  \psi(\bar p/q, s_{q+1} + c  ).
\ee
Hence, for some $i\in\{1, \dots, n\}$, we have
\be
\nonumber
	\lvert x_i -  p_i/q \rvert \geq e^{-(1+r^i)(s_{q+1} +c)} \geq \tfrac{e^{-(1+ \max\{r^i\}) c}}{2^{2n+2} n!}q^{-(1+r^i)},
\ee
and we see that \textbf{Bad}$(\cal{F}) \subset $ \textbf{Bad}$_{X}^{\bar r}$.
Remarking that a supset of a winning set is also a winning set finishes the proof.
\end{proof}

Although the Simplex Lemma is folklore,
we want to give the proof of \cite{KristensenEtAl} for the sake of completeness.

\begin{proof}[Proof of the Simplex Lemma \ref{SimplexLemma}]
Let $D\subset \R^n$ be a convex subset of volume less than $1/(n! k^{n+1})$.
Assume by contradiction that there are $n+1$ rational vectors $\bar p_i/q_i \in D \cap R_k$ which are not contained in an affine hyperplane.
These vectors span a simplex $S$ which is contained in $D$ by convexity of $D$.
Moreover, the volume of $S$, vol$(S)\neq 0$, is given by 
\be
\nonumber
	\text{vol}(D) \geq \text{vol}(S) =  \frac{1}{n!} \lvert \det( 
	\begin{pmatrix} 
			1 & \bar p_1^T/q_1 \\
			1 & \bar p_2^T/q_2 \\
			1 & \bar p_3^T/q_3 
	\end{pmatrix}) \rvert \geq \frac{1}{n!}\frac{1}{q_1 \cdots q_{n+1}}> \frac{1}{n!k^{n+1}},
\ee
which is a contradiction and finishes the proof.  
\end{proof}


\subsection{\textbf{Bad}$_{\C^n}^{\bar r}$}
\label{BadCn}
Let $\Z[i]$ be the ring of Gaussian integers in $\C$.
For $n\geq 1$, let again $\bar r \in \R^n$ with $r^1,\dots, r^n \geq 0$ such that $\sum r^i =1$.
Denote by  \textbf{Bad}$_{\C^n}^{\bar r}$  the set of points $\bar x = (x_1,\dots ,x_n)\in \C^n$ 
for which there  is a positive constant $c(\bar x)>0$
such that 
\be
\nonumber
 	\max_{i=1, \dots, n} \lvert  qx_i - z_i \rvert^{1/r^i}  \geq c(\bar x) \cdot  \lvert q\rvert^{-1},
\ee
for every $z_1, ..., z_n, q \in \Z[i]$, $q \neq 0$. 

Let $\Omega= \C^n \times (0, \infty)$ and define the monotonic function $ \psi =  \psi_{\bar r}$ by the box
\be
\nonumber
	\psi( \bar x, t  ) \equiv B(x_1, e^{-(1+r^1)t} )\times \dots \times B(x_n, e^{-(1+r^n)t} ).
\ee

Note that \cite{DodsonKristensen} showed \textbf{Bad}$_{\C^2}^{(1/n, \dots, 1/n)}$ to be a winning set for Schmidt's game.
We want to show the following stronger result.

\begin{theorem}
\textbf{Bad}$_{\C^2}^{\bar r}$ is absolute $\psi$-winning with respect to the collection $\cal{S}$ of  complex lines.
\end{theorem}
	
\noindent 
By Theorem \ref{Dimension} (which we will see is applicable),
for any nonempty open set $U \subset  \C^2$, we have 
for the Lebesgue measure $\mu$  on $\C^2$,
\be
\label{DimensionComplex}
	 \text{dim}( \textbf{Bad}_{\C^2}^{\bar r} \cap U) \geq d_{\mu}(U).
\ee

Let $\mu_i$ satisfy a power law on $X_i \subset \C$, $i= 1, 2$, and set $X = X_1 \times  X_2 \subset \C^2$ with the product measure $\mu=\mu_1 \times \mu_2$.
Then \cite{KristensenEtAl} showed that \textbf{Bad}$_{\C^2}^{\bar r} \cap X$ is of Hausdorff-dimension dim$(X)$. 
In fact, in this case, $\mu$ is an absolutely decaying measure 
(compare with Lemma \ref{ProductDecaying}, modified with respect to complex affine subspaces),
we can modify the proof below and show that \textbf{Bad}$_{\C^2}^{\bar r}\cap X$ is absolute $\psi$-winning with respect to $\cal{S}$ in $X$.
Moreover, \eqref{DimensionComplex} holds for sets $U\subset X$ and with respect to the product measure $\mu$.

For simplicity and since all the arguments can be carried out analogously to the proof of Theorem \ref{BadN} with respect to the complex setting,
we restrict to the full space $X=\C^2$ and only sketch the proof.

\begin{proof}[Sketch of the Proof.]
For $n \in \Lambda\equiv \N_{\geq 2}$ define the resonant set
\be
\nonumber
	R_n \equiv \{ (z_1/q, z_2/q) \in \C^2 : z_1, z_2, q \in \Z[i], 0< \lvert q\rvert < n \}
\ee
with size $s_n \equiv \log(n)$, which gives a nested and discrete family $\cal{F}$.

We remark that implicitly in the proof of \cite{KristensenEtAl}, Theorem 17,  the following analogue of the Simplex Lemma is contained.

\begin{lemma}
There exists $\bar l_* > 0$ such that, 
if $D=B(x_1, r_1) \times B(x_2, r_2)$ is a box with $r_1 r_2 <  e^{-\bar l_*} n^{-3}$,
then $D \cap R_n$ is contained in a complex line $L$. 
\end{lemma}

\noindent Thus, for any $l_* > \bar l_*/3 $ with $\bar l_*$ as above, we have that $\psi(\bar x, \log(n) + l_*)$ is a box with radii $r_1$, $r_2$ satisfying
\be
\nonumber
	r_1 r_2 = e^{-(1+r^1)( \log(n) + l_*)}\cdot e^{-(1+r^2)( \log(n) + l_*)} < e^{-\bar l_*} n^{-3},
\ee
and we see that $\psi(\bar x, s_n + l_*) \cap R_n$ is contained in a complex line.
This shows that $\cal{F}$ is locally contained in $\cal{S}$, the set of complex lines, with $n_*=1$.

Moreover $(\Omega, \psi)$ is $b_*$-diffuse with respect to $\cal{S}$ and $d_*$-separating for some $b_*> 0$ and $d_* = \log(3)/(1+\min\{r^1, r^2\})$.
Thus, Theorem \ref{DiffuseProposition} implies that  $(\Omega, \psi)$ is strongly $\bar b_*$-diffuse with respect to $\cal{F}$, where $\bar b_*=l_*+b_* +d_*$,
as well as that  \textbf{Bad}$_{\C^2}^{\psi}( \cal{F})$ is absolute $(\psi, \bar b_*)$-winning with respect to $\cal{S}$.
Finally,  it is readily checked that \textbf{Bad}$_{\C^2}^{\psi}(\cal{F}) \subset $ \textbf{Bad}$_{\C^2}^{\bar r}$.
\end{proof}


\subsection{\textbf{Bad}$_{\Z_p^n}^{\bar r}$}
\label{BadZn}
Let $p$ be a prime number , $\lvert \cdot \rvert_p$ the $p$-adic absolute value 
and $\Z_p $ be the $p$-adic integers in the $p$-adic field $ \Q_p$.
For $n\geq 1$, let again $\bar r \in \R^n$ with $r^1,\dots, r^n \geq 0$ such that $\sum r^i =1$.
Because of the different properties of the $p$-adic field, 
we need to adjust the definition of  badly approximable $p$-adic vectors.
For further details, we refer to \cite{KristensenEtAl} and references therein. 
Let \textbf{Bad}$_{\Z_p^n}^{\bar r}$ be the set of points $\bar x = (x_1,\dots ,x_n)\in \Z_p^n$ 
for which there exists a positive constant $c(\bar x)>0$
such that 
\be
\nonumber 
 	\max_{i=1, \dots, n} \lvert  x_i - \tfrac{z_i}{q} \rvert_{p}^{1/(1+r^i)}  
	\geq c(\bar x) \max\{ \lvert z_1\rvert, \dots,\lvert z_n\rvert,\lvert q\rvert \}^{-1},
\ee
for all $(z_1, ..., z_n) \in \Z^n$ and $q \in \N$. 
Let $d(x,y)\equiv \lvert x-y\rvert_p$ be the $p$-adic metric on $\Z_p$.
For $(\bar x, t) \in \Q_p^n \times (0, \infty)$ consider the box
\be
\nonumber
	\bar \psi(\bar x, t) \equiv B(x_1, e^{-(1+r^1)t}) \times \dots \times B(x_n, e^{-(1+r^n)t}).
\ee 
For $n=2$, it was already shown by \cite{KristensenEtAl} that 
(a slightly different version of) \textbf{Bad}$_{\Z_p^2}^{\bar r}$ is of Hausdorff-dimension $2$.
We show the following stronger result.

\begin{theorem}
\textbf{Bad}$_{\Z_p^2}^{\bar r}$ is absolute $\psi$-winning with respect to the collection $\cal{S}$ of $p$-adic lines in $\Z^2_p$ and thick;
 that is, for any nonempty open set $U\subset \Z_p^2$, we have
\be
\nonumber
	\text{dim}(\textbf{Bad}_{\Z_p^2}^{\bar r} \cap U)=2.
\ee
\end{theorem}

\begin{proof}[Sketch of the proof]
\noindent 
As previously, let $\Lambda\equiv \N_{\geq 2}$ and for $n\in \Lambda$ define the resonant set
\be
\nonumber
	R_n \equiv \{ (z_1/q, z_2/q) \in \Q_p^2 : z_1, z_2 \in \Z, q \in \N \text{ such that } 
	 \max\{\lvert z_1 \vert,\lvert z_2 \vert,\lvert q \vert  \} < n \}
\ee
with the size $s_n \equiv \log(n) $.
For the nested discrete family $\cal{F} = (\Lambda, R_n, s_n)$ we show that for $X=\Z^2_p$, $(\Omega, \psi)$ is strongly $b_*$-diffuse with respect to $\cal{F}$.

Let $m\equiv \mu \times \mu$, where $\mu$ is the normalized Haar-measure on $\Q_p$.
Hence, $\mu(\Z_p)=1$ and $m(B(x_1, r_1) \times B(x_2, r_2)) = p^{-(t_1+t_2)}$ for $p^{-t_i}\leq r_i \leq p^{-t_i+1}$ and $t_i\in \N$, $i=1,2$.
In particular, for $\omega=(\bar x, t) \in \Omega$  we have $p^{-4}e^{-3t}\leq m(\psi(\omega)) \leq e^{-3t} $.
Thus,  $(\Omega, \psi, \mu)$ satisfies a power law with respect to the exponent $\tau=3$.

Again, we  remark that implicitly in the proof of \cite{KristensenEtAl}, Theorem 18, the following analogue of the Simplex Lemma is contained.

\begin{lemma}
\label{pAdicSimplex}
Let $D\subset \Z_p^2$ be a box of measure $m(D) < 1/(6 n^3)$.
Then there exists a  $p$-adic line $L$ such that $R_n \cap D \subset L$.
\end{lemma}

\noindent Thus, let $l_*> \log(6)/3$.
For $\omega=(\bar x, t ) \in \Omega$ and $s_n$ with $s_n \leq t$ we have $m(\psi(x,t+l_*)) < 1/(6 n^3)$
and Lemma \ref{pAdicSimplex} implies that 
$ R_n \cap \psi(x, t + l_*) \subset L$, 
for an affine $p$-adic line $L$.
This shows that $\cal{F}$ is locally contained in $\cal{S}$, the collection of $p$-adic lines, with $n_*=1$.

Next, we claim that $(\Omega, \psi)$ is $b_*$-diffuse with respect to $\cal{S}$ for $b_*>0$ sufficiently large.
Therefore, note that, as shown in \cite{KristensenEtAl}, for $b_*>0$ sufficiently large,
a geometric argument implies that any number of disjoint boxes $\psi(\bar x_i, t + b_*) \subset \psi(\omega)$, $\bar x_i \in \Z^2_p$, 
intersecting a $p$-adic $L$ is bounded above by $C \cdot e^{b_*(1+\max\{r^1, r^2\})}$, 
where $C$ is independent of $b_*$ and $t$.
Using that $(\Omega, \psi, \mu)$ satisfies a power law with respect to the exponent $\tau=3$,
for $b_*>0$ sufficiently large, 
there exists a collection of disjoint boxes $\psi(\bar x_i, t + b_*) \subset \psi(\omega)$, $\bar x_i \in \Z^2_p$,
whose number exceeds the one of its boxes intersecting $L$ (independently from $t$).
If we take such a box $\psi(\bar x_i, t + b_*) \subset \psi(\omega)$, $\bar x_i \in \Z^2_p$, not intersecting $L$,
then $\psi(\bar x_i, t + 2 b_*)$ is disjoint from $\psi(L, t + 2b_*)$ (if $b_*$ is sufficiently large).
This shows the above claim.

Since $(\Omega, \psi)$ is moreover $d_*$-separating for $d_* \leq \log(3)/(1+\min\{r^1, r^2\})$,
Theorem \ref{DiffuseProposition} shows that $(\Omega, \psi)$ is strongly $(2b_* + l_* + d_*)$-diffuse with respect to $\cal{F}$
and, moreover, that \textbf{Bad}$(\cal{F})$ is absolute $\psi$-winning with respect to $\cal{S}$.
Furthermore, by Theorem \ref{Dimension} and since (MSG1-2) is satisfied,
for any open set $U=B(z_1, e^{-t_1})\times B(z_2, e^{-t_2)} \cap \Z_p^2$, $z_1, z_2 \in \Z_p$, we have
\be
\nonumber
	\text{dim}(\textbf{Bad}(\cal{F}) \cap U)= d_{\mu}(U) =2.
\ee

Finally, let $\bar x\in \textbf{Bad}_{\Z_p^2}^{\psi}(\cal{F})$ and $(z_1/q, z_2/q) \in \Q^2$ 
with $\max\{\lvert z_1 \rvert, \lvert z_3 \rvert, \lvert q\rvert\}=n$.
There exists $c(x)<\infty$ such that $\bar x \not \in \bar \psi( R_n, s_n+c(x)) \supset \bar \psi((z_1/q, z_2/q), s_n+c(x))$.
Hence, for some $i\in\{1,2\}$ we have 
\be
\nonumber
	\lvert x_i - z_i/q \rvert_p > e^{-(1+r^i)(s_n + c(x)) } \geq e^{-(1+\max\{r^1, r^2\})(c_* + \log(4)+c(x))} n^{-(1+r^i)}.
\ee
Therefore, \textbf{Bad}$(\cal{F}) \subset \textbf{Bad}_{\Z^2_p}^{\bar r}$ which finishes the proof.
\end{proof}

\paragraph{II. Examples from Dynamical Systems.}


\subsection{The Bernoulli shift $\Sigma^+$}
\label{Bernoulli}
For $n\geq 1$, let $\Sigma^+=\{0,\dots ,n\}^{\N}$ be the set of one-sided sequences in symbols from $\{0,\dots ,n\}$.
Let $T$ denote the shift and let $d^+$ be the metric given by $d^+(w, \bar w)\equiv e^{-\min\{ i \geq 1: w(i) \neq \bar w(i) \}}$ for $w \neq \bar w$ and $d(w, w)\equiv0$.
 
Fix a periodic word $\bar w \in \Sigma^+$ of  period $p\in \N$ and consider the set 
\be
\nonumber
	S_{\bar w} = 
	\{ w\in \Sigma^+ : \exists \ c=c(w)< \infty \text{ such that }T^k w \not \in B(\bar w,  2^{-(p+c +1)}) \text{ for all } k \in \N\}.
\ee

\begin{theorem}
$S_{\bar \omega}$ is absolute winning (in the sense of McMullen) and of Hausdorff-dimension $\log(n)$ (and in fact thick).
\end{theorem}

\begin{remark} In particular, the intersection $\bigcap S_{\bar w}$ over all periodic words $\bar w \in \Sigma^+$ is $(B_1,1)$-absolute winning. 
Note that the Morse-Thue sequence $w$ in $\{0,1\}^{\N}$ is a particular example of a word in $\bigcap_{\bar w} S_{\bar w}$.
In fact, $w$ does not contain any subword of the form $WWa$ where $a$ is the first letter of the subword $W$;
for details and more general words in $\bigcap_{\bar w} S_{\bar w}$, we refer to the author's earlier work \cite{SchroederWeil}.
\end{remark}

\begin{proof}
For $k \in \N$ and  $ w_k \in \{0,..,n\}^k $,  let $\bar w_k \in \Sigma^+$ denote the word $\bar w_k=w_k \bar w$.
Let $\Lambda \equiv \N_0$ and consider the resonant sets
\be
\nonumber
	R_0 = \{\bar w\}, \ \ \ \ R_{k} = \{\bar  w_l \in \Sigma^+ : w_l \in \{1,..,n\}^l, l\leq k  \}) \cup R_0, \text{ for } k \in \N
\ee 
which we give the size $s_k = p+k+1$.
Then, $\cal{F} = (\N_0, R_k, s_k)$ is nested and discrete.

Let $\bar w_m$ and $\tilde w_m \in R_m$ be distinct.
By definition of $\bar w_m$ and $ \tilde w_m$ there exists $i \in \{1, \dots, m+p\}$ such that $\bar w_m(i) \neq \tilde w_m(i)$;
hence  
\be
\nonumber
	d^+(\bar w_m, \tilde w_m) \geq  e^{-(p+m+1)} =   e^{-s_m}
\ee
and we are given the special case \eqref{Distinct}.
Moreover, $(\Sigma^+, d^+)$ is $\beta$-diffuse for $\beta=1$.
Proposition \ref{Sufficient} shows that \textbf{Bad}$(\cal{F})$ is absolute-winning.

Moreover the probability measure $\mu = \{1/n, \dots , 1/n\}^{\N}$ satisfies 
\be
\nonumber
	 \mu(B(w, e^{-(t+1)})) = n^{-t} = n e^{- \log(n)(t+1)},
\ee
where $t\in \N$.
Hence, $(\Sigma^+ \times \N, B_1, \mu)$ satisfies a  $(\log(n),n,n)$-power law
and we see that \textbf{Bad}$(\cal{F})$ is of Hausdorff-dimension $\log(n)$ (and thick) by Theorem \ref{Dimension}.

Finally, we have \textbf{Bad}$(\cal{F}) = S_{\bar w}$.
In fact, $d^+(T^{k-1} w, \bar w) \leq e^{-(p+c+1)}$ for some $c\in \N$ if and only if  $w(k)\dots w(k+p+c)=\bar w(1)\dots \bar w(p+c)$.
Thus, for $w_k=w(1)\dots w(k)$ and $\bar w_k=w_k\bar w$ we have
$d^+(w, \bar w_k) \leq e^{-(p+k+c +1)}$ if and only if $w \in B(\bar w_k, e^{-(s_k +c)}) \subset \psi_1(R_k, s_k+c)$.
\end{proof}


\subsection{Toral endomorphisms $E_{\cal{M}, \cal{Z}}$}
\label{ToralEndo}
For the motivation, further generalizations and consequences of the following result, we refer to \cite{BFK} and references therein.
For $n\in \N$, let $\cal{M}=(M_k)$ be a sequence of real matrices $M_k\in GL(n,\R)$ 
and $\cal{Z}=(Z_k)$ be a sequence of $\tau_k$-separated%
\footnote{ That is, for every $y_1, y_2 \in Z_k$ we have $d(y_1,y_2) \geq \tau_k>0$.}
subsets of $\R^n$.
Define 
\be
\nonumber
	E_{\cal{M}, \cal{Z}} \equiv \{x \in \R^n  : \exists\  c=c(x) >0 \text{ such that } d(M_kx, Z_k) \geq c \cdot \tau_k \text{ for all } k \in \N\},
\ee
where $d$ is the Euclidean distance.
The sequence $\cal{M}$ is \emph{lacunary}, if for $t_k= \lVert M_k \rVert_{op}$ (the operator norm) 
we have $\inf_{k\in \N} \tfrac{t_{k+1}}{t_k} \equiv \lambda >1$.
The sequence $\cal{Z}$ is \emph{uniformly discrete}, if there exists $\tau_0>0$ such that every set $Z_k$ is $\tau_0$-separated.
Under the assumption that $\cal{M}$ is lacunary and $\cal{Z}$  is uniformly discrete, 
\cite{BFK} showed that if $X$ is the support of an absolutely $\delta$-decaying measure, 
then $E_{\cal{M}, \cal{Z} } \cap X$ is a winning set in $X$ for Schmidt's game.

Using similar arguments for the proof, we want to consider the following weaker condition in our weaker setting:
In fact, assume that, independently of $t\in \R^+$, we have
\be
\label{Count}
	\lvert \{ k \in \N : \log(t_k/\tau_k) \in (t-b, t] \} \rvert \leq \varphi(b), \ \ \text{ for all } b>0,  
\ee
for some function $\varphi : \R^+ \to \R^+$.
Note that if $\cal{M}$ is lacunary and $\cal{Z}$  is uniformly discrete, then \eqref{Count} holds for the function $\varphi(b) \leq b/\log(\lambda)$.

Let again $\cal{S}$ denote the set of affine hyperplanes in $\R^n$ and recall that the Lebesgue measure is absolutely $(1, c_0)$-decaying 
(see Lemma \ref{ProductDecaying}).

\begin{theorem}
Let $X \subset \R^n$ be the support of an absolutely $\delta$-decaying measure $\mu$,
let  $\cal{M}$ and  $\cal{Z}$ be as above satisfying \eqref{Count} for a function $\varphi(b) \leq e^{\bar \delta b}$, with $\bar \delta <\delta$.
Then, for every $n_*\in \N$ and $L_*\geq 0$ there is $b_*=b_*(n_*, L_*, \delta, \bar \delta)$ such that $(\Omega, B_1)$ is $(b_*, n_*, L_*)$-diffuse with respect to $\cal{F}$ defined below, where  \textbf{Bad}$_X(\cal{F}) \subset E_{\cal{M}, \cal{Z}} \cap X$.
\end{theorem}

\noindent
In particular,  $E_{\cal{M}, \cal{Z}}  \cap X$ is  $B_1$-weakly-winning by Theorem \ref{Winning} 
and, in view of  Propositions \ref{Preserving} and \ref{Intersection},
the same is true for its image under any bi-Lipschitz map 
and for every finite intersection $\cap_{i=1}^{n_*}E_{\cal{M}_i, \cal{Z}_i}$ of such families $(\cal{M}_i, \cal{Z}_i)$.

Moreover, if $\mu$ satisfies moreover  a power law with respect to the exponent $\tau$,
then $E_{\cal{M}, \cal{Z}} \cap X$ 
is of Hausdorff-dimension $\tau$ (and in fact thick) 
by Theorem \ref{Dimension}.

\begin{proof}
Let $v_k \in \R^n$ be the unit vector such that $\lVert M_k v_k \rVert = t_k$ 
and if $V_k \equiv \{M_kv_k\}^{\perp}$ is the subspace orthogonal to $M_kv_k$, let $W_k \equiv M_k^{-1}(V_k)$.
Then, for $k \in \N$ and $z \in Z_k$  we  define the subsets 
\be
\nonumber
	Y_k(z)\equiv (M_k^{-1}(z) + W_k)  \cap M_k^{-1}(B(z, \tau_k/4) ).
\ee
Set $s_k \equiv \log(\tau_k/t_k) + \log(12)$, which we reorder such that $s_k \leq s_{k+1}$, so that we obtain a discrete set of sizes.
For $k\in \Lambda \equiv \N$ let  the resonant set $R_k$ be given by
\bea
\nonumber
	R_k &\equiv& \{ x \in Y_l(z_l): z_l \in Z_l \text{ and }  \log(t_l / \tau_l)   \leq s_k \} 
	\\ \nonumber
		&=& \{ x \in Y_l(z_l): z_l \in Z_l \text{ and }  \frac{\tau_l}{t_l} \geq \frac{\tau_k}{t_k} \}, 
\eea
which gives a nested and discrete family $\cal{F}=\{\N, R_k, s_k\}$.

Note that for all $x\in \R^n$ we have $\lVert x \rVert \geq \lVert M_kx \rVert /t_k$.
Hence, for distinct points $z_1$, $z_2 \in Z_k$, $Y_k(z_1)$ and $Y_k(z_2)$ are subsets of parallel affine hyperplanes and  we have
\bea
\label{Disjoint1}	
 	\lVert Y_k(z_1) - Y_k(z_2) \rVert &\geq&
	\lVert M_k^{-1} (B(y_1, \tau_k/4)) -M_k^{-1} (B(y_2, \tau_k,/4)) \rVert 
	\\ \nonumber
	&\geq& \frac{\tau_k - 2 \tau_k/4}{t_k}  =  \frac{\tau_k }{2t_k}  \geq 6 e^{-s_k},
\eea
since $Z_k$ is $\tau_k$-separated.
Given a closed ball $B=B(x, 3 e^{-t}) \subset \R^n$ with $x\in X$, for every $k\in \N$ with $s_k \leq t$, 
it follows from \eqref{Disjoint1} that at most one of the sets $Y_k(y)$, $y\in Z_k$, can intersect $B$.
Moreover, for $b>0$, the number of $k \in \N$ with $s_k \in (t-b, t]$
is bounded by $\varphi(b)$ by  \eqref{Count}. 
Thus, there exist at most $N=\lfloor \varphi(b) \rfloor$ affine hyperplanes $L_1, \dots, L_N \in \cal{S}$ such that
\be
\nonumber
	B \cap  R(t,b)  \subset \bigcup_{i=1}^N L_i.
\ee
Since $\mu$ is absolutely $(\delta, c_{\delta})$-decaying with respect to $\cal{S}$, and $(\Omega, B_1)$ is $\log(3)$-separating,
we have for $B=B(x, e^{-t})$ and $s\geq 0$ that
\bea
\nonumber
	\mu(B \cap \cal{N}_{e^{-(t + s)}}\big(R(t,b) \big) )
		&\leq& \sum_{i=1}^N  \mu( B \cap \cal{N}_{e^{-(t +s)}}(L_i)) 
		\\ \nonumber
		&\leq& \varphi(b) \cdot c_{\tau} e^{-\delta s} \mu( B) \equiv f(b, s) \mu(B).
\eea
Note that, since $\varphi(b) \leq e^{ \bar \delta b}$ with  $\bar \delta <\delta$, 
for all $n_*\in \N$ and $L_* \geq 0$,
there exists a $b_*= b_*(n_*, L_*, \delta, \bar \delta)$ such that
$f(n_*(b + L_*) + \log(3) , b-  2\log(3)) \leq c_0 <1$ for all $b>\bar b_*$.
Thus, we showed that $(\Omega, B_1, \mu)$ is $f$-decaying with respect to $\cal{F}$ and $f$ is $(\log(3), b_*, n_*, L_*)$-decaying.

Moreover, $(\Omega, B_1)$ is $\log(3)$-separating with respect to $\cal{F}$.
Hence, by Proposition \ref{DecayingMeasure},
$(\Omega, B_1)$ is $( b_* + \log(3),n_*, L_*)$-diffuse with respect to $\cal{F}$.

Finally, let $x \in $ \textbf{Bad}$_X(\cal{F})$, that is,  there exists $c<\infty$ such that
\be
\nonumber
	d(x, Y_k(y))\geq e^{-(s_k+c)} \geq   e^{-c - \log(12)}  \tau_k/t_k  \equiv \bar c \tau_k/t_k
\ee
  for every $k\in \N$ and $y\in Z_k$.
Assume that $M_k x \in B(y, \tau_k/4)$.
Then,  
\be
\nonumber
	 x \in \cal{N}_{\bar c \tau_k/t_k}(M_k^{-1}(y) + W_k)^C \cap M_k^{-1}(B(y, \tau_k/4))  
\ee
and we can write the vector $v = x- M_k^{-1}(y)$ as $v=w + \tilde c \tau_k/t_k v_k$ with $w \in W_k$ and $\tilde c \geq \bar c$.
Hence, since $M_kW_k$ is orthogonal to $M_kv_k$, 
\be
\nonumber
	d(M_k x, y) 
	= \lVert M_k v \rVert 
	= \lVert M_k w + \tilde c\ \tfrac{\tau_k}{t_k} M_k v_k \rVert \geq \tilde c\ \tfrac{\tau_k}{t_k} \lVert M_k v_k \rVert \geq \bar c\tau_k,
\ee
so that $M_k x \not \in B(y_k, \bar c \tau_k)$.
This shows that \textbf{Bad}$_X(\cal{F}) \subset E_{\cal{M}, \cal{Z}} \cap X$, finishing the proof.
\end{proof}


\subsection{The geodesic flow in CAT(-1)-spaces}
\label{CAT(-1)}
We discuss this example  in more details.
If $GZ$ denotes the space of geodesic rays in a proper geodesic CAT(-1) metric space $Z$, then the semigroup $\R^+$ acts on $GZ$ 
via the geodesic flow $(g^s)$ which itself acts by reparameterization,
\be \nonumber
	g^s(\gamma)(t) = \gamma(t+s).
\ee
Given a collection $\cal{C}$ of (convex) sets in $Z$ we can ask about the rays which avoid
contractions of or have bounded penetrations in neighborhoods of these sets.
The behavior of  penetration lengths of geodesic rays in convex subsets of $Z$ leads to a model of Diophantine approximation in CAT(-1)-spaces, 
developed by Hersonsky, Parkkonen and Paulin in \cite{HersonskyPaulin,HersonskyPaulin2,Parkkonen2}, allowing applications to metric number theory,
as well as \cite{MayedaMerrill}.
With respect to the visual metric $d_o$ (where $o$ is a base point),
we thereby translate our problem to the compact metric space $(\partial_{\infty}Z, d_o)$ and, 
since $d_o$ is a metric on the set of asymptotic rays, 
we induce suitable resonant sets  in $\partial_{\infty}Z$ related to the collection $\cal{C}$.

We begin by introducing the setting and stating the main results of this subsection.
In Subsubsection \ref{DANegative} we introduce the model of Diophantine approximation and relate the model to our setting and results.
In Subsubsection \ref{PattersonMeasure}, we discuss on the question of the Hausdorff-dimension and on the required conditions.
In order to keep the exposition readable, we will skip all of the main proofs until Subsubsection \ref{ProofsCAT(-1)}.

\subsubsection{Main Results}
For a general reference and further details we refer to \cite{Bridson}.
In the following, $(Z,d)$ denotes a proper geodesic CAT(-1) metric space 
and, for a convex subset $Y \subset Z$, $\partial_{\infty}Y$ its visual boundary, that is, the set of  equivalence classes of asymptotic rays in $Y$.
Equip $\bar Z \equiv Z \cup \partial_{\infty}Z$ with the cone topology.
Given two points $x, y\in \bar Z$ we denote by $[x,y]$ the unique geodesic segment from $x$ to $y$.
For three points $o, x, y \in \bar Z$, let 
\be
\nonumber
	(x,y)_o \equiv  \frac{1}{2}( d(o,x) + d(o, y) - d(x,y))
\ee
be the Gromov-product at $o$ 
and for $\xi, \eta \in \partial_{\infty}Z$, let $(\xi, \eta)_o \equiv \lim_{t\to \infty} (\gamma_{o, \xi}(t), \gamma_{o, \eta}(t))_o$
be the extended Gromov-product at $o$, 
where $\gamma_{o, \xi} \equiv [o, \xi] $. 
For $o\in Z$, we define $d_o:   \partial_{\infty}Z \times \partial_{\infty}Z \to [0, \infty)$ by  $d_o(\xi, \xi )\equiv 0$ and  for $\xi \neq \eta$ by 
\be
\nonumber
	d_o(\xi, \eta) \equiv e^{-(\xi, \eta)_o},
\ee
called the \emph{visual metric} at $o$. 
Then $(\partial_{\infty}Z, d_o)$ is a compact metric space.
\footnote{
Note that the visual distance at a point $o\in Z$ is comparable to the Hamenst\"adt metric with respect to a horoball $H_0$:
For every compact subset $K$ of $\partial_{\infty}Z - \partial_{\infty}H_0$, there exists a constant $c_K>0$ 
such that for all $\xi, \eta \in K$,
\be
\nonumber
	c_K^{-1} d_o (\xi, \eta) \leq d_{H_0}(\xi, \eta) \leq c_K d_o(\xi, \eta);
\ee
see \cite{HersonskyPaulin}, Lemma 2.3.
We therefore focus only on the visual distance in our settings, which can however,
up to further requirements,  
be replaced by the Hamenst\"adt metric.
}

For $\xi \in \partial_{\infty}Z$ and $y\in Z$, the \emph{Busemann function} $\beta = \beta_{\xi,y} : Z  \to \R$ (with respect to $y$) is defined by
\be
\nonumber
	\beta(x) \equiv \lim_{t \to \infty} d(x, \gamma_{y, \xi}(t)) - t,
\ee
which is  continuous and convex on $Z$ and $\beta(y)=0$.
The level sets of $\beta_{\xi,y}$ are called \emph{horospheres} at $\xi$ and
the sublevel sets are called \emph{horoballs} at $\xi$ (with respect to $y$).

For technical reasons, let $t_0>0$ be a sufficiently large constant determined below.
Now, given a base point $o \in Z$,
assume we are given a countable collection of closed convex sets $\cal{C} = \{ C_m \subset Z : m \in \N \}$ 
such that the collection of distances $\{ d_m \equiv d(o, C_m): m \in \N \} \subset (t_0, \infty)$ is a discrete set.
Remarking that $Z$ is a $\delta_0$-hyperbolic space for some $\delta_0>0$ (see \eqref{Tripod}), we will consider the following three cases simultaneously:
\begin{itemize}
\item[1.] $\cal{C}_1  = \{C_m\}$ is a collection of pairwise disjoint horoballs based at $\partial_{\infty}C_m \equiv \xi_m$.
\item[2.] $\cal{C}_2 = \{C_m\}$ is a collection of convex sets with $\lvert \partial_{\infty}C_m\rvert \geq 1$  which is \emph{$(2 \delta_0, T)$-embedded}; that is,
we have that diam$(\cal{N}_{2 \d_0}(C_i )\cap \cal{N}_{2 \delta_0}(C_j)) \leq T$ for $i\neq j$.
\item[3.] $\cal{C}_3 = \{x_m\}$ is a collection of $\tau_0$-separated points $x_m$ in the  hyperbolic space $Z=\H^{n+1}$.
\end{itemize}
Note that Case $1.$ is in fact covered by Case $2$. but treated explicitly as an interesting special case.

In the first two cases, we obtain a collection of nonempty sets $\cal{C}^{\infty}_i \equiv \{ \partial_{\infty}C_m\}$ in $ \partial_{\infty}Z$
which we will see is disjoint.
For the third case, let $\cal{C}^{\infty}_3 \equiv \{ \xi^{\infty}_m \} $ be the collection of the boundary projections of $x_m$ with respect to $o$,
that is $\xi^{\infty}_m\equiv \gamma_{o,x_m}(\infty) \in S^n = \partial_{\infty}\H^{n+1}$.
By abuse of notation, $\xi^{\infty}_m$ is also denoted by $\partial_{\infty}C_m$ in the following.
The following result on the distribution of $\cal{C}_{\infty}$ in $\bar X$ is crucial.

\begin{proposition}
\label{DistributionCAT(-1)}
Let $l_1 \equiv \delta_0$ and $l_2 \equiv T + 2\delta_0$.
Then, for the respective cases, we have 
\begin{itemize}
\item[1.] $d_o(\xi_i, \xi_j) > e^{-l_1} e^{- \max\{d_i, d_j\}}$,
\item[2.]  $d_o(\partial_{\infty}C_i, \partial_{\infty}C_j ) > e^{-l_2} e^{- \max\{d_i, d_j\}}$,
\end{itemize}
for $i\neq j$.
Moreover, there exists a constant $c_0=c_0(\tau_0)$ such that, for every $b>0$ and every ball $B=B_{d_o}(\xi, 2e^{-t})$,
\begin{itemize}
\item[3.] $\lvert \{ \xi_m^{\infty} \in B : d_m \in (t-b, t]\}\rvert \leq c_0 \ b$.
\end{itemize}
\end{proposition}

Given the three collections $\{ d_m^i \equiv d(o, C_m) : m \in \N \}$, $i=1,2,3$,
which we relabel to define the set of sizes $s^i_m \equiv d_m^i$ and reorder such that $s^i_m \leq s^i_{m+1}$.
For $m \in \Lambda_i \equiv \N$ let 
\be
\nonumber
	\bar R^i_m \equiv \{ \xi \in \partial_{\infty}C_j : \partial_{\infty}C_j \in \partial_{\infty}\cal{C}
	\text{ such that } e^{-d_j} \geq e^{-s_m}   \},
\ee
which gives a nested and discrete family $\cal{F}_i= (\N, R^i_m, s^i_m)$.
Given a closed subset $X \subset \bar X\equiv \partial_{\infty}Z$, set $\Omega \equiv X \times (t_0, \infty)$.
If moreover every set $\partial_{\infty} C_m \in \cal{C}^{\infty}_2$ is closed (hence compact), note that a point $\xi \in$ \textbf{Bad}$_X^{B_1}(\cal{F}_i)$ 
if and only if there exists a constant $c = c(\xi)>0$ such that,
for every $\partial_{\infty} C_m \in \cal{C}_i^{\infty}$,
\be
\label{BadInterpretation}
	d_o(\xi, \partial_{\infty} C_m) > c\ e^{- d(o,C_m)}.
\ee
Assuming that $X$ satisfies suitable diffusion properties in $\bar X$ with respect to the collections $\cal{C}^{\infty}_i$,
we obtain our main result using Proposition \ref{DistributionCAT(-1)}.

\begin{theorem}
\label{GeodesicFlow}
For Case 1. and for Case 2. if every set $C_i \in \cal{C}_2$ is a geodesic line,
assume that $X$ is $\beta$-diffuse.
Then \textbf{Bad}$_X^{B_1}(\cal{F}_i)$ is absolute-winning (in the sense of McMullen).

For Cases $1, 2$. assume that $(\Omega, B_1) $ is $b_*$-diffuse with respect to $\cal{C}_i^{\infty}$.
Then \textbf{Bad}$_X^{B_1}(\cal{F}_i)$ is absolute $B_1$-winning with respect to $\cal{C}_i^{\infty}$ (in particular Schmidt winning).

If $X$ is the support of a locally finite Borel measure such that $(\Omega, B_1, \mu)$ satisfies a power law with respect to the exponent $\tau$,
then, for every $n_*\in \N$ and $L_*\geq 0$, there is $b_*=b_*(n_*, L_*, \tau, \tau_0)$ such that $(\Omega, B_1)$ is $(b_*, n_*, L_*)$-diffuse with respect to $\cal{F}_3$.
\end{theorem}

\noindent
In particular,  $\textbf{Bad}_X^{B_1}(\cal{F}_3)$ is a $B_1$-weakly-winning by Theorem \ref{Winning} 
and, in view of  Propositions \ref{Preserving} and \ref{Intersection},
the same is true for its image under any bi-Lipschitz map 
and for any finite intersection $\cap_{i=1}^{n_*}\textbf{Bad}_X^{B_1}(\cal{F}^i_3)$ of such families $\cal{F}^i_3$.
Moreover, $\textbf{Bad}_X^{B_1}(\cal{F}_3)$ is of Hausdorff-dimension $\tau$ (and in fact thick) by Theorem \ref{Dimension}.

We remark that the first case has been considered by \cite{MayedaMerrill} (as well as \cite{Aravinda,Dani,Dani2,MayedaMerrill,McMullen,Schmidt} in stronger and more specific settings than ours) 
in the setting of proper geodesic $\delta$-hyperbolic metric spaces
where they used a similar definition of badly approximable points using the size of the shadows of the disjoint horoballs.
In fact, note that our proof of Case $1.$ works equally well in their weaker setting.\\
Moreover, Case $3.$ also holds if $Z$ is a manifold of pinched negative curvature.

Before we relate our setting to the model of Diophantine approximation due to Hersonsky, Parkkonen and Paulin,
we want to point out the following dynamical interpretation,
which is one of the main reasons to require a CAT(-1) rather than a $\delta$-hyperbolic-space $Z$.

\begin{lemma} \label{DA-Dynamical}
Let $t_0, l_0>0$ be sufficiently large constants.
Given $C=C_m \in \cal{C}_{i}$ with $d(o, C) \geq t_0$ and $\xi \in \partial_{\infty}Z$,
we have
\begin{itemize}
\item[1.] 	$\gamma_{o, \xi}([t, t+ l]) \subset C$, 
\item[2.] 	$\gamma_{o, \xi}([t, t+l]) \subset \cal{N}_{2\delta_0}(C)$, 
\item[3.] 	$\gamma_{o, \xi}(t) \in B(x_m, e^{-l})$,
\end{itemize}
for a suitable time $t>t_0$ and length $l>l_0$, if and only if 
\begin{itemize}
\item[1.] 	$d_o(\xi, \xi_m) \leq \bar c \ e^{-l_0/2} \cdot e^{-d(o,C)}$,
\item[2.]	$d_o(\xi, \partial_{\infty}C) \leq \bar c\ e^{-l} \cdot e^{-d(o,C)}$, 
\item[3.]	$d_o(\xi, \xi_m^{\infty}) \leq \bar c\ e^{-l} \cdot e^{-d(o, C)}$,
\end{itemize}
where $\bar c>0$ is a universal constant.
\end{lemma}

\noindent Hence, in view of \eqref{BadInterpretation}, 
for $i=1,2,3$, if every $\partial_{\infty}C_m \in \cal{C}^{\infty}_i$ is closed, we obtain that
\be
\nonumber
	\textbf{Bad}_X^{B_1}(\cal{F}_i) = S_i,
\ee
for the following sets $S_i$. 
\begin{itemize}
\item[1.] $S_1 \equiv  \{\xi \in X : 
	\exists\  l= l(\xi) < \infty $
	 such that the lengths%
	\footnote{ Note that since $C_m$ is convex, $\gamma_{o, \xi}(\R^+) \cap C_m$ is the image of a connected geodesic segment.}
	$L(\gamma_{o, \xi}(\R^+) \cap C_m) \leq l $ for all horoballs $C_m \in \cal{C}_1 \}$,
\item[2.] $S_2 \equiv  \{\xi \in X : 
	\exists\  l= l(\xi) < \infty $
	 such that the lengths
	$L(\gamma_{o, \xi}(\R^+) \cap \cal{N}_{2\d_0}(C_m)) \leq l $ for all totally geodesic sets $C_m \in \cal{C}_2 \}$,
\item[3.] $S_3 \equiv \{\xi \in X : 
	\exists\  c= c(\xi) >0$ such that $\gamma_{o, \xi}(\R^+) \cap B(x_m, c)=\emptyset$ for all points $x_m \in \cal{C}_3 \}$.
\end{itemize}

\begin{remark}
In view of Lemma \ref{LemmaClosing} below, we can in fact consider the $\e$-neighborhoods $\cal{N}_{\e}(C_m)$ of $C_m$ in $S_2$, for $\e>0$.
Moreover,  a geodesic $\gamma_{o, \xi}$, $\xi \in S_1$, has bounded penetration lengths in the collection of horoballs $\cal{C}_1$ 
if and only if it avoids the same collection of uniformly shrinked horoballs;
see Lemma \ref{RelationDeepPenetration} for a  precise statement.
\end{remark}


\subsubsection{A  model of Diophantine approximation in negatively curved spaces.}
\label{DANegative}
Let $\Gamma \subset I(Z)$ be a discrete subgroup of the isometry group $I(Z)$ of $Z$.
Note that every isometry $\varphi\in I(Z)$ extends to a homeomorphism on $\partial_{\infty}Z$.
The \emph{limit set} $\Lambda\Gamma$ of $\Gamma$ is the compact subset $\overline{ \Gamma. x} \cap \partial_{\infty}Z$ of $\partial_{\infty}Z$,  for any $x \in Z$.
If $\Lambda\Gamma$ contains at least two points, then $\cal{C}\Gamma$ denotes the convex hull of $\Lambda\Gamma$.

Recall that a subgroup $\Gamma_0$ of $\Gamma$ is called \emph{convex cocompact} 
if $\Lambda \Gamma_0$ contains at least two points and the action of $\Gamma_0$ on the convex hull $\cal{C}\Gamma_0$ has compact quotient.\\
We call $\Gamma_0$ \emph{bounded parabolic} if $\Gamma_0$ is the stabilizer of a parabolic fixed point $\xi_0 \in \Lambda\Gamma$, 
and if there exists a horoball $C_0$ at $\xi_0$ such that the action of $\Gamma_0$ on $\partial C_0$ has compact quotient.
Note that up to considering the CAT(-1)-space $Z \cap \Lambda\Gamma$ instead of $Z$, 
our definition agrees with the classical definition of bounded parabolic fixed points in the hyperbolic space; see \cite{Ratcliffe}. \\ 
Finally, we call $\Gamma_0$  \emph{almost malnormal} if $\Lambda \Gamma_0$ is precisely invariant,
that is,
if for all $\varphi \in \Gamma-\Gamma_0$ we have $\varphi. \Lambda\Gamma_0 \cap \Lambda \Gamma_0 = \emptyset$.

Now, let $\Gamma_{i} \subset  \Gamma$, $i=1,2$, be  almost malnormal subgroups in $\Gamma$ of infinite index and without elliptic elements.
We treat the following three cases simultaneously:
\begin{itemize}
\item[1.] Let $\Gamma_1$ be bounded parabolic and let $C_1$ be the horoball as in the definition.
\item[2.] Let $\Gamma_2$ be convex-cocompact and let  $C_2 = \cal{C}\Gamma_2$ be the convex hull of $\Gamma_2$,
where we assume that either
\begin{itemize}
\item[a)] $C_2$ is a geodesic line, or, 
\item[b)] every image $\varphi. \Lambda \Gamma_{2} \subset \Lambda\Gamma$, $[\varphi] \in \Gamma/\Gamma_2$, is contained in a metric sphere in $\partial_{\infty}Z$ 
		(with respect to $d_o$).
\end{itemize}
\item[3.] Let $\Gamma_3$ be the identity element of $\Gamma$, and, for $ x\in Z=\H^{n+1}$,  take $C_3=\{x\}$ in the following.
\end{itemize}
Note that since $\Gamma_i$ is almost malnormal for the Cases $1, 2.$ and $\Gamma$ is without elliptic elements for Case $3.$, we have $\Gamma_i = $ Stab$_{\Gamma}(C_i)$.
\\

\noindent \textbf{Example.}
Let $C_2$ be a totally geodesic submanifold of dimension  $m+1$ in $\H^{n+1}$, 
the hyperbolic ball model, and $o=0$ be the center of $\H^{n+1}$.
Hence, $C_2$ is a subspace isometric to the hyperbolic space $\H^{m+1}$ and the boundary of this subspace is a metric sphere  (with respect to the angle metric $d_o$) of dimension $m$.
Since $\Gamma_2$ is almost malnormal and convex cocompact, we have $m< n$.
Hence, $\partial_{\infty}C_2  = \Lambda \Gamma_2$ and every image $\varphi. \Lambda\Gamma_2$ are contained in metric spheres.
\\

For the respective cases, $i=1,2,3$, given again $t_0>0$, denote the data by
\be
\nonumber
	\cal{D}_i=(Z, \Gamma, C_i, o, t_0).
\ee
For $r=[\varphi] \in\Gamma/\Gamma_{i}$ we define
\be
\nonumber
	D_i(r) = d(o, \varphi (C_i))
\ee
which does not depend on the choice of the representative $\varphi$ of $r$. 

\begin{remark}
For $r=[\varphi] \in \Gamma/\Gamma_{i}$,  
let $\cal{S}_o(\varphi (C_i)) \subset \partial_{\infty}Z$ be the shadow of the set $\varphi( C_i)$ with respect to the base point $o$.
Using \eqref{Compare}, Lemma \ref{Decreasing} and \ref{Entering} below, 
one can show that if $D_i(r)$ is sufficiently large,
the size (diameter with respect to $d_o$) of the shadows $\cal{S}_o(\varphi C_i)$ is comparable to the quantitiy $e^{-D_i(r)}$. 
We therefore consider the \emph{approximation function} $f_i(r) = e^{D_i(r)}$ as a renormalization of the size of the shadows.
\end{remark}
 
Note that the set $\{D_i(r) : r\in   \Gamma/\Gamma_{i}\}$ is discrete and unbounded.

\begin{lemma}
\label{DiscreteCosets} 
For every $D \geq 0$ there are only finitely many elements $r\in \Gamma/\Gamma_{i}$ 
such that $D_i(r) \leq D$ and there exists an $r\in \Gamma/\Gamma_{i}$ such that $D_i(r)>D$.
\end{lemma}

\begin{proof} For the second case, the proof follows from Lemma 3.1 and 3.2 in \cite{Parkkonen2}
with the difference that we do not consider the stabilizer of $o$ in $\Gamma$ (which is trivial in our assumption and only a finite subgroup in general).
The arguments of the proof also work for the first case.
The third case follows since $\Gamma$ is discrete and $\Gamma_3$ is of infinite index in $\Gamma$.
\end{proof}

\noindent Now, for $i=1,2,3$ and for $\xi \in \Lambda\Gamma - \Gamma. \Lambda\Gamma_{i}$ define the \emph{approximation constant}
\be
\nonumber
	c_i(\xi) = \liminf_{r=[\varphi] \in \Gamma/\Gamma_{i}: D_i(r)>t_0} e^{D_i(r)} d_o(\xi, \varphi.\Lambda \Gamma_i),
\ee
where we replace $\varphi.\Lambda\Gamma_i$ by $\varphi(x)_{\infty} \equiv \gamma_{o, \varphi(x)}(\infty)$  in the third case.
If $c_i(\xi)=0$ then $\xi$ is  called \emph{well approximable}, 
otherwise it is called \emph{badly approximable} (with respect to $\cal{D}_i$).
Define the set of badly approximable limit points  by
\be
\nonumber
	\textbf{Bad}(\cal{D}_i) = \{\xi \in \Lambda\Gamma - \Gamma.\Lambda \Gamma_{i} : c_i(\xi)>0 \} \subset \Lambda\Gamma.
\ee

Consider the collections $\cal{C}_i \equiv \{ \varphi(C_i): r=[\varphi] \in \Gamma/\Gamma_i $ with $D_i(r)>t_0 \}$, $i=1,2,3$,
and note that $\cal{C}_i$ is \emph{$(2 \delta_0, T)$-embedded}.
In fact, this follows easily for Case $3.$ since, by discreteness of $\Gamma$, $\cal{C}_3$ is in fact $\tau_0$-separated for some $\tau_0>0$.
For Case $2.$ we refer to \cite{HersonskyPaulin2} and remark that the proof works similarly for Case $1.$ 
In the first case, we will therefore assume, after shrinking $C_1$, that the images $\varphi (C_1)$, $[\varphi] \in \Gamma/\Gamma_1$ are pairwise disjoint.
Using Lemma \ref{DiscreteCosets}, we thus established the setting of the previous subsubsection for the corresponding cases.
Again, in view of \eqref{BadInterpretation}, we have for $X=\Lambda\Gamma$ and $\bar X= \partial_{\infty}Z$ that
\be
\nonumber
	\textbf{Bad}(\cal{D}_i) = \textbf{Bad}_{\Lambda\Gamma}^{B_1}(\cal{F}_i) = S_i.
\ee
Thus, as a corollary of Theorem \ref{GeodesicFlow}, we obtain the following.

\begin{corollary}
\label{GeodesicFlow2}
If the limit set $\Lambda\Gamma$ of $\Gamma$ is $\beta$-diffuse 
for the Cases $1.$ and $2a)$,
then \textbf{Bad}$(\cal{D}_i)$ is absolute-winning (in the sense of McMullen).

For the Case $2b)$, if $(\Lambda\Gamma \times (t_0, \infty), B_1)$ is $b_*$-diffuse with respect to the collection $\cal{S}$ of metric spheres in $\Lambda\Gamma$,
then \textbf{Bad}$(\cal{D}_2)$ is absolute $B_1$-winning with respect to $\cal{S}$.

If $\Lambda\Gamma$ is the support of a locally finite Borel measure satisfying a power law with respect to the exponent $\tau$,
then for any bi-Lipschitz map $F: S^n \to F(S^n)$,
$F(\textbf{Bad}(\cal{D}_3))$ is weakly $B_1$-winning in $F(\Lambda\Gamma)$ and of Hausdorff-dimension  $\tau$ (and in fact thick  in $F(\Lambda\Gamma)$).
\end{corollary}

\begin{remark}
Recall that if $X= \H^{n+1}$ and $\Gamma$ is a non-elementary finitely generated Kleinian group, 
then $\Lambda\Gamma \subset S^n$ is uniformly perfect; see \cite{JarviVuorinen}.
In particuar, $\Lambda\Gamma$ is $\beta$-diffuse for some $\beta>0$ by Lemma \ref{UniformlyPerfect}.
\\
For Case $2.$ we refer to Corollary \ref{DiffuseSpheresLimitSet} below, and for Case $3.$ to the next subsubsection.
 \end{remark}

\subsubsection{A measure on $\Lambda\Gamma$.}
\label{PattersonMeasure}
Let $X= \H^{n+1}$ be the hyperbolic ball model and let $o=0$ be the center.
Note that the visual distance $d_o$ is bi-Lipschitz equivalent to the angle metric on the unit sphere $S^n = \partial_{\infty}\H^{n+1}$.
Hence, if $\Gamma$ is of the \emph{first kind}, that is $\Lambda\Gamma=\partial_{\infty}\H^{n+1}$, 
then $\Lambda\Gamma= S^n$ is $\beta$-diffuse and $b_*$-diffuse with respect to $\cal{S}$, for $\beta=b_*>\log(3)$, 
and the Lebesgue measure on $S^n$ satisfies a power law with respect to the visual metric $d_0$ and the exponent $n$.
More generally, recall  that the \emph{critical exponent} of a discrete group $\Gamma \subset I(\H^{n+1})$ is given by
\be
\nonumber
	\delta(\Gamma) \equiv \inf \{s>0 : \sum_{\varphi \in \Gamma} e^{-s d(x, \varphi(x)) } < \infty\},
\ee
for any $x\in \H^{n+1}$.
Associated to $\Gamma$, there is a canonical measure, the \emph{Patterson-Sullivan} measure $\mu_{\Gamma}$, which is a $\delta(\Gamma)$-conformal probability measure supported on $\Lambda\Gamma$. 
For a precise definition we refer to \cite{Nicholls}.
If $\Gamma$ is non-elementary and convex-cocompact,
then $\delta(\Gamma)$ equals the Hausdorff-dimension of $\Lambda\Gamma$;
in particular, the Patterson-Sullivan measure $\mu_{\Gamma,o}$ (at $o$) satisfies a power law with respect to the exponent $\delta(\Gamma)$.
There are various further results concerning the Patterson-Sullivan measure.
Here, we point out the following.

Regarding Case $1.$, \cite{MayedaMerill} showed that if $\Gamma$ is a non-elementary geometrically finite Kleinian group, 
the set of limit points which correspond to geodesics starting in $o$ and projecting to bounded geodesics in $\H^{n+1}/\Gamma$ 
has dimension $\delta(\Gamma)$.
In particular, $S_1 = \textbf{Bad}(\cal{D}_1)$ contains this set and is thus of dimension $\delta(\Gamma)$.

For the second case, let $\cal{H}(\Gamma) \equiv \{ S \cap \Lambda \Gamma : S$ is a sphere in $S^n$ of codimension at least $1$$\}$
which contains the set $\cal{S}$.
A finite Borel measure $\nu$ on $S^n$ is called \emph{$\cal{H}(\Gamma)$-friendly},
if $\nu$ is Federer and if $(\Lambda\Gamma \times (t_0, \infty), B_1, \nu)$ is absolutely $(\delta, c_{\delta})$-decaying with respect to $\cal{H}(\Gamma)$.

\begin{theorem}[\cite{StratmannUrbanski}, Theorem 2]
\label{SU}
For every non-elementary convex cocompact discrete group $\Gamma \subset I(\H^{n+1})$ (without elliptic elements),
such that $\Lambda\Gamma$ is not contained in a finite union of elements of $\cal{H}(\Gamma)$,
the Patterson-Sullivan measure $\mu_{\Gamma,o}$ is $\cal{H}(\Gamma)$-friendly.
\end{theorem}

\noindent The Theorem is in fact true for a set of $\mu_{\Gamma,o}$-neglectable subsets (for details and the definition we refer to \cite{StratmannUrbanski}) 
and the requirements of Case $2.$ can be weakened.

As a corollary of Proposition \ref{AbsDecMeasure} we obtain the following.

\begin{corollary}
\label{DiffuseSpheresLimitSet}
Let $\Gamma$ be as in Theorem \ref{SU}.
Then
$(\Lambda \Gamma \times (t_0, \infty), B_1)$ is $b_*$-diffuse with respect to $\cal{H}(\Gamma)$ and hence with respect to $\cal{S}$ for some $b_*>0$ sufficiently large.
\end{corollary}

\noindent Hence, for Case $2.$,  $S_2=$ \textbf{Bad}$(\cal{D}_2)$ is absolute $B_1$-winning with respect to $\cal{S}$ by Corollary \ref{GeodesicFlow2}.
Moreover, since $\mu_{\Gamma,o}$ satisfies a power law, we see that $S_2$ is thick, by Theorem \ref{Dimension}. 

Summarizing, we have the following.
\begin{corollary}
In our setting of Subsubsection \ref{DANegative},
if $\Gamma$ is non-elementary geometrically finite Kleinian group in Case $1.$
or  as in Theorem \ref{SU} in Case $2.$, then $\textbf{Bad}(\cal{D}_i)$ is of Hausdorff-dimension $\delta(\Gamma)$ for $i=1,2$.
In fact, for any nonempty open set $U \subset \Lambda\Gamma$, 
$\textbf{Bad}(\cal{D}_2) \cap U$ is of Hausdorff-dimension $d_{\mu}(\Lambda\Gamma)= \delta(\Gamma) = $ dim$(\Lambda\Gamma)$.
\end{corollary}

\subsubsection{Proofs.} \label{ProofsCAT(-1)}

First, note that every CAT(-1) space is a (tripod) $\delta_0$-hyperbolic space for some $\delta_0>0$,
which implies that for all $ o, x, y\in Z $ we have
\be
\label{Tripod}
	p \in [o, x], q \in [o, y] \text{ with } d(o,p)=d(o,q) \leq (x,y)_o  \implies d(p,q) \leq \delta_0,
\ee
and \eqref{Tripod} also holds for $x,y \in \partial_{\infty}Z$ as well. 
Moreover, $Z$ is a (Gromov) $\delta_0$-hyperbolic space (we may assume the same $\delta_0$)
so that, given a geodesic triangle with vertices $x, y, z \in \bar Z $, every edge of the triange lies in the $\d_0$-neighborhood of the two other ones.
Finally, there exists a $\k>0$ (depending only on $\delta_0$), such that for all $o\in Z$ and $\xi, \eta \in \partial_{\infty}Z$,
\be 
\label{Compare}
	0 \leq d(o, [\xi, \eta]) - (\xi, \eta)_o \leq \k.
\ee

We start with the proofs of Proposition \ref{DistributionCAT(-1)} and Theorem \ref{GeodesicFlow}.

\begin{proof}[Proof of Proposition  \ref{DistributionCAT(-1)}]
We start with Cases $1.$ and $2.$
Given $\eta \in \partial_{\infty}C_{1} $, $\bar \eta \in \partial_{\infty}C_{2} $, where $C_1$ and $C_2 \in \cal{C}_i$, 
set $D\equiv \max\{ d(o, C_1), d(o, C_2)\} $, and
assume that 
\be
\label{A1}
	d_o(\eta, \bar\eta)  = e^{-(\eta, \bar \eta)_o} \leq 
	e^{- (D  + l_i)}.
\ee
Equivalently, we have $(\eta, \bar \eta)_o \geq D +l_i$ and by \eqref{Tripod}, we see that 
\be
\label{A2}
	d(\gamma_{o,\eta}( D + l_i), \gamma_{o,\bar \eta}( D + l_i)) \leq  \delta_0.
\ee
\emph{Case 1:}
By definition, $l_1=   \delta_0$ and hence, 
\be
\nonumber
	d(\gamma_{o, \eta}( D +  \delta_0), \gamma_{o,\bar  \eta}( D +  \delta_0)) \leq  \delta_0.
\ee 
On the other hand, 
both the points $\gamma_{ \eta}( D +  \delta_0)$ and $\gamma_{\bar \eta}( D + \delta_0)$ 
are contained in the horoballs $C_1 $ and $C_2$ respectively, 
at distance at least $\delta_0$ to the boundaries of the respective horoballs.
Therefore, if the horoballs $ C_1$ and $C_2$ are disjoint, then
\be
\nonumber
	d(\gamma_{o,\eta}( D + \delta), \gamma_{o,\bar \eta}( D +\delta_0)) \geq 2\delta_0,
\ee 
which is a contradiction. 
Hence, $C_1=C_2$ and $\{\eta\} = \{\bar \eta\}$.
\\
\emph{Case 2:}
By definition, $l_2=  T+ 2\delta_0$, and hence, 
\be
\nonumber
	d(\gamma_{ o,\eta}( D + T+ 2\d_0), \gamma_{o,\bar \eta}( D + T+ 2d_0)) \leq  \delta_0.
\ee 
Let $o_j$ be the projection of $o$ on the closed convex set $C_j$ with $d(o, o_j) = d(o,C_j)$, $j=1,2$.
For any point $x$  on the ray $\gamma_{o, \eta }$ at distance $d(o, x) > D + \delta_0$ and $y$ on $[o, o_1]$, we have
$d(x, y) \geq d(o, x) - d(y, o) > \delta_0$.
Since $Z$ is a Gromov $\delta_0$-hyperbolic space, we see that $x$ cannot belong the $\delta_0$-neighborhood of the segment $[o, o_1]$
and must therefore be contained in $\cal{N}_{\d_0}([o_1, \eta])$.
Since $C_1$ is convex, $o_1 \in C_1$ and $\eta \in \partial_{\infty}C_1$, we have $[o_1, \eta] \subset C_1$.
Thus, we showed
\be
\nonumber
	\gamma_{o,\eta}( [D + \delta_0, D + T+ 2\delta_0])  \subset \cal{N}_{ \delta_0} ([o_1, \eta]) \subset \cal{N}_{ \delta_0} (C_1),
\ee 
and the analogous results is true for $\gamma_{o,\bar \eta}$.
Therefore, by convexity of the distance function and by \eqref{A2}, we have 
\be
\nonumber
	\gamma_{o, \eta}( [D+\delta_0, D + T+ 2\d_0]) \subset \cal{N}_{2 \delta_0} (C_2),
\ee 
and hence
\be
\nonumber
	\gamma_{o,\eta}( [D+\delta_0, D + T + 2\d_0]) \subset \cal{N}_{ \d_0} (C_1) \cap \cal{N}_{2 \d_0} (C_2).
\ee 
In particular, since $\cal{C}_2$ is $(2\d_0, T)$-embedded and
\be
\nonumber
	\text{diam}(\cal{N}_{2\d_0}(C_1) \cap \cal{N}_{2 \d_0}(C_2))\geq L(\gamma_{o,\eta}( [D+\d_0, D + T + 2\d_0])) = T + \d_0 >T,
\ee
we must have $C_1=C_2$ and $ \eta, \bar \eta \in \partial_{\infty} C_1$.

Now, for \emph{Case 3.,}
we switch to the hyperbolic space.
For a subset $M\subset S^n$ and $0\leq a \leq\bar a$, consider the truncated cone of $M$ with respect to $o$,
\be
\nonumber
	M(a,\bar a) \equiv \{ \gamma_{o, \xi} (t) \in \H^{n+1} : \xi \in M, a\leq t \leq \bar a \}.
\ee
Fix $b>0$, a ball $B=B_{d_o}(\eta, 2e^{-t})$ (with $t\geq b$) and 
note that a point $\xi_m^{\infty}$ with $t-b< d(o, x_m) \leq t$ lies in $B$ if and only if  $x_m \in B(t-b, t)$.
It therefore suffices to estimate the number of $x_m \in B(t-b,t)$ which we denote by $G(\eta, t, b)$.

First, we claim  that $B(t-b, t)$ is contained in
the $(\delta_0 + 2\log(2))$-neighborhood of the geodesic segment $\gamma_{o, \eta}((t-b,t])$.
To see this, note that for any point $\xi \in B$,
we have $(\xi, \eta)_0 \geq - \log(d_o(\xi, \eta)) \geq t - \log(2)$ and hence, by \eqref{Tripod} 
$d(\gamma_{o, \xi}(s), \gamma_{o, \eta}(s)) \leq \d_0$, for all $s \leq t - \log(2)$.
For $t-\log(2) \leq s \leq t$ we have 
\bea
\nonumber
	d(\gamma_{o, \xi}(s), \gamma_{o, \eta}(s)) &\leq& d(\gamma_{o, \xi}(s), \gamma_{o, \xi}(t-\log(2))) + \delta_0 + d(\gamma_{o, \eta}(s), \gamma_{o, \eta}(t-\log(2))) 
	\\ \nonumber
	&\leq& \delta_0 + 2\log(2),
\eea
concluding the claim.

Clearly, since $\H^{n+1}$ is of constant sectional curvature, 
there exists a universal constant $C>0$ such that the hyperbolic volume of $\cal{N}_{\d_0 + 2\log(2)}(\gamma_{o, \eta}((t-b,t]))$ is bounded by $C\cdot b$.
Since moreover $\cal{C}_3$ is $\tau_0$-separated for some $\tau_0>0$,
it also follows that there exists a constant $\bar c = \bar c(\tau_0)>0$ such that the   (hyperbolic) volume of every ball $B(x_m, \tau_0/2)$ is at least $\bar c$.
Thus, we conclude that $G(\eta,r,b)\leq C/\bar c \cdot b$, finishing the proof.
\end{proof}

\begin{proof}[Proof of Theorem \ref{GeodesicFlow}]
For Case $1.$ and Case $2.$, when every set $C_m$ is a geodesic line, assume that $X$ is $\beta$-diffuse.
For the second case we need to remark that for distinct points $\eta$, $\bar \eta \in \partial_{\infty}C$, for a geodesic line $C \in \cal{C}_2$,
then by \eqref{Compare} we have  
\be
\nonumber
	d_o(\eta, \bar \eta) = e^{-(\eta, \bar \eta )_o} \geq e^{- d(o, [\eta, \bar \eta])} = e^{-d(o, C)}.
\ee 
Hence, using this remark and Proposition \ref{DistributionCAT(-1)}, we see that the special case \ref{Distinct} is satisfied for $c_i = e^{-l_i}$ and $\sigma=1$.
Moreover, $(\bar X, d_o)$ is compact so that $(\bar \Omega, B_1)$ is $\log(3)$-contracting.
Thus, Proposition \ref{Sufficient} implies that  $(\Omega, B_{1})$ is strongly $\bar b_*$-diffuse with respect to $\cal{F}_i$, where $\bar b_*=  l_i + \log(3) + \log(2) + \beta$.
In addition, \textbf{Bad}$_X^{B_1}(\cal{F}_i)$ is absolute-winning (in the sense of McMullen) in the respective cases.

For Case $2.$ in general,
it follows from Proposition \ref{DistributionCAT(-1)}  that, given a ball $B=B_{d_o}(\xi, e^{-(t + l_* + \log(2)})$, $\xi \in X$, $t>t_0$,
then for every size $s_m\leq t$ we have that $B \cap R_m $ is either empty or equals the set $B \cap \partial_{\infty}C_j$ for some set $C_j \in \cal{C}_2$.
We showed that $\cal{F}$ is locally contained in $\cal{C}_2^{\infty}$ for $n_*=1$.
Since $(\Omega, B_1)$ is $b_*$-diffuse with respect to $\cal{C}_2^{\infty}$,
Theorem \ref{DiffuseProposition}
shows that $(\Omega, \psi)$ is strongly $\bar b_*$-diffuse with respect to $\cal{F}_2$ where $\bar b_*=l_2 + \log(2) + \log(3) + b_*$,
and, moreover, that \textbf{Bad}$_X^{B_1}(\cal{F}_2)$ is absolute $B_1$-winning with respect to $\cal{C}_2^{\infty}$.
Finally,  the same is true for Case $1.$, and Theorem \ref{Winning} concludes the first two cases.

For Case $3.$, 
we want to show that $(\Omega, B_1)$ is $(b_*, n_*, L_*)$-diffuse with respect to $\cal{F}$ for every $n_*\in \N$ and $L_*\geq 0$.
In fact, assume that $X = $ supp$(\mu)$ for a locally finite Borel measure on $S^n$ which satisfies a power law with respect to $\tau$.
Given a ball $B=B(\xi, e^{-t})$, $\xi \in X$, $t>t_0$,
consider the set of boundary projections $\xi_m^{\infty} \in \cal{C}_3^{\infty}$ with $d(o, x_m)\in (t-b, t]$  and $\xi_m^{\infty} \in 2B=B(\xi, 2e^{-t})$.
This is precisely the set $2B \cap R^3(t, b)$ and 
Proposition \ref{DistributionCAT(-1)} implies that $\lvert 2B \cap R^3(t,b) \rvert$ is bounded by $ c_0 \cdot b$. 
Moreover, a ball $B(\xi_{m}^{\infty}, e^{-(t+s)})$ with $\xi_{m}^{\infty}\not\in 2B$ cannot intersect $B$.
Hence, we have
\bea
\nonumber
	\mu(B \cap \cal{N}_{e^{-(t+s)}}(R^3(t,b) )) 
		&\leq& \cup_{\xi_m^{\infty} \in 2B \cap R^3(t,b)}  \mu( B_{d_o}(\xi_m^{\infty} , e^{-(t +s)})) 
		\\ \nonumber
		&\leq&  C b \cdot c_2 e^{-\tau(t +s)}
		\\ \nonumber
		&\leq& \tfrac{C c_2}{c_1} b \cdot e^{-\tau s} \mu(B) \equiv f(b,s) \mu(B),
\eea
which shows that $(\Omega, B_1, \mu)$ is $f$-decaying with respect to the family $\cal{F}_3$.
Clearly, there exists $b_*= b_*(n_*, L_*, \tau, \tau_0)> 2 \log(3)$ sufficiently large such that 
the function $f$ satisfies $f( n_*(b + L_*) + \log(3), b- 2\log(3)) \leq c_0 <1$ for all $b>b_*$.
Since $R^3(t,b)$ is a discrete set for all $t,b>0$,
 $(\Omega, B_1)$ is $\log(3)$-separating with respect to $\cal{F}_3$.
 Thus, Proposition \ref{DecayingMeasure} concludes that $(\Omega, B_1)$ is $(\bar b_*, n_*, L_*)$-diffuse with respect to $\cal{F}$, where $\bar b_* =  b_* + \log(3)$.

Since \eqref{Separating} and (MSG1-2) are satisfied, we have  
dim(\textbf{Bad}$_X^{B_1}(\cal{F}_3)\cap U) = d_{\mu}(U) = \tau$ from Theorem \ref{Dimension}, for $U \subset X$ open.
This finishes the proof.
\end{proof}

We will make use of the following results.

\begin{lemma}[\cite{PaulinParkkonen2}, Lemma 2.1]
\label{Decreasing}
Let $x$, $y\in Z$ and for $z\in Z\cup \partial_{\infty}Z$ let $\gamma=[x,z]$.
Then, for all $t\in [0, d(x,z)]$, 
\be
\nonumber
	d(\gamma(t), [y,z]) \leq  \tfrac{1}{2} e^{d(x,y)-t}.
\ee
\end{lemma}

\noindent 
If $\e>0$ and $\alpha$ is a geodesic segment, let $\cal{N}_{\e}(\alpha)$ be the closed $\e$-neighborhood of $\alpha$ which is itself convex.
As a consequence of Lemma \ref{Decreasing}, 
we prove that a ray which penetrates in the $D$-neighborhood of a geodesic segment for a sufficiently long time must also penetrate in its $\e$-neighborhood.

\begin{lemma} 
\label{LemmaClosing}
Let $D \geq \e>0$.
Let $\gamma$ and $\alpha$ be two geodesics in $X$ such that $d(\gamma(-L) , \alpha) \leq D$ and $d(\gamma(L), \alpha) \leq D$, 
where $L\geq 2(D - \log(\e))$.
Then there exists a constant $c=c(D, \e) \leq D - \log(\e)$ such that  $\gamma([-L+c, L-c]) \subset \cal{N}_{\e}(\alpha)$.
\end{lemma}

\begin{proof}
First, consider the case when $\gamma$ and $\alpha$ do not intersect.
Let $p$ and $q \in \alpha$ be the closest points of $a=\gamma(L)$ and $b=\gamma(-L)$ respectively at distance at most $D$ on $\alpha$.
We subdivide the quadrilateral $(a,b, p,q)$ in two geodesic triangles $(a, b, p)$ and $(b, p,q)$ with a connecting geodesic $\tilde \gamma=[b, p]$.
Note that $\tilde L \equiv d(b, p) \geq 2L - D$.
For $t\in [0, \tilde L]$, we let  $b_t \in \gamma$ and $q_t \in \alpha$ be the closest points of $\tilde \gamma(t)$ on $\gamma$ and $\alpha$ respectively.
Let $t_0= D -  \log(\e) $.
From Lemma \ref{Decreasing} we have
$d(\tilde \gamma(t_0), q_{t_0}) \leq e^{-t_0}e^{D}/2= \e/2$,
as well as, since $L\geq 2(D - \log(\e))$,
\be
\nonumber
	d(\tilde \gamma(t_0), b_{t_0}) \leq \tfrac{1}{2} e^{-(\tilde L-t_0)} e^{D} \leq \tfrac{1}{2} e^{-2L+D+ \log(\e)} \leq \tfrac{\e}{2}.
\ee
Thus, $d(b_{t_0}, \alpha) \leq \e$.
Note that $d(\gamma(L), b_{t_0})\leq t_0$ by properties of the closest point map.
In the same way, we define $a_{t_0}$ for the two geodesic triangles $(a, b, q)$ and $(a, p,q)$.
Similarly, we obtain that also $d(a_{t_0}, \alpha) \leq \e$ with $d(\gamma(-L), a_{t_0})\leq t_0$.
Therefore, we see by convexity of the distance function that
$\gamma([-L+t_0, L-t_0]) \subset [a_{t_0}, b_{t_0}] \subset \cal{N}_{\e}(\alpha)$.

The case when $\gamma$ and $\alpha$ intersect follows from the same arguments (and is simpler).
\end{proof}

\begin{lemma}[\cite{PaulinParkkonen2}, Lemma 2.9]
\label{Entering}
Let $C_0$ be a horoball in $Z$ and $o\in Z-C_0$.
Then, for two geodesic rays starting in $o$ and entering in $C_0$ at $x$ and $\bar x$ respectively,
we have
\be 
\nonumber
 	d(x, \bar x) \leq 2\log(1+\sqrt{2}) \equiv c_0 .
\ee
\end{lemma}

If $\tau\geq 0$ and $C_0=\beta^{-1}((-\infty,0])$ is a (closed) horoball with respect to the Busemann function $\beta$, let 
$C_0[\tau] \equiv \beta^{-1}((-\infty,-\tau]) = \{x \in C_0 : d(x, \partial C_0) \geq \tau\}\subset C_0$ 
denote the horoball shrinked by the factor $\tau$.
Let $o\in X-C_0$ and assume that for $\xi \in \partial_{\infty}X$ the ray $\gamma_{o, \xi}$ enters in $C_0$.
Define the \emph{shrinking parameter} of $\xi$ by $s(\xi) = \sup \{\tau \in [ 0, \infty] : \gamma_{o,\xi} \cap C_0[\tau] \neq \emptyset \}$.
Then the ray $\gamma_{o, \xi}$ penetrates the horoball $C_0$ for a long time if and only if it enters deeply into $C_0$, that is, its shrinking parameter is large.

\begin{lemma}
\label{RelationDeepPenetration}
Let $o\in Z - C_0$. 
Assume that for $\xi \in \partial_{\infty}Z$ the ray $\gamma_{o, \xi}$ enters in $C_0$ at time $t \geq 0$ and leaves at time $t+p$, $0<p<\infty$.
Let $s\geq 0$ be the shrinking parameter of $\xi$.
Then
\be
\nonumber
	2s - c_0 \leq p \leq 2s + 2c_0.
\ee
\end{lemma}

\begin{proof}
Let $C_0$ be based at the point $\eta \in \partial_{\infty}Z$, $\eta\neq \xi$, and let $d_o = d(o, C_0)\geq 0$ such that $\gamma_{o, \eta}(d_o)\in \partial C_0$.
Note that the function $s\mapsto \beta \circ \gamma_{o, \xi}(s)$ is  continuous and convex.
Hence, there exists a point $\xi_s\equiv \gamma_{o, \xi}(t+p_1)$ on $\partial C_0[s] = \beta^{-1}(-s)$.
By Lemma \ref{Entering}, we have
$d(\gamma_{o, \xi}(t), \gamma_{o, \eta}(d_o)) \leq c_0$ as well as 
$d(\xi_s, \gamma_{o, \eta}(d_o+s)) \leq c_0$.
Note that for all $\tau \geq 0$,  $\gamma_{o, \eta}(d_o+\tau)$ is the closest point of $o$ to $\partial C_0[\tau]$.
Hence, $d_o \leq t \leq d_o + c_0$ as well as 
\be
\nonumber	
	d_o+s \leq t+p_1 \leq d_o + c_0 + s.
\ee
Starting with the point $\tilde o = \gamma_{o, \xi}(t+p) \in \partial C_0$ with $d_{\tilde o}= d(\tilde o, C_0)=0$,
we obtain in the same way by Lemma \ref{Entering} that $s \leq p_2 \equiv p-p_1 \leq s + c_0 $.
Thus,
\bea
\nonumber
	2s- c_0 \leq 2s +d_o - t &\leq& p_1 +p_2 = p 
	\\ \nonumber
	&\leq& d_o-t + 2 s+  2c_0 \leq 2 s+  2c_0,
\eea
which finishes the proof.
\end{proof}

Finally, we are able to prove Lemma \ref{DA-Dynamical}.

\begin{proof}[Proof of  Lemma \ref{DA-Dynamical}]
For the first case, 
given a horoball $C$ based at $\eta$ and a point $\xi \in \partial_{\infty}Z$,
assume  that $\gamma_{o, \xi}([t, t+l] )\subset C$.
We may assume that $t$ is the entering and $t+ l$ the exiting time.
Then from Lemma \ref{RelationDeepPenetration}, $l \leq 2s + 2c_0$, 
where $s$ denotes the shrinking parameter of $\xi$ in $C$.
Moreover, if $x$ is the closest point of $o$ on $[\xi, \eta]$, 
we claim that $d(o, x) \geq d(o, C[s]) - 2\d_0$.
Assuming the claim,  we have
\be
\nonumber
	d(o, x) \geq d(o, C[s]) -2\d_0 = d(o, C) + s \geq d(o,C) + l/2 - c_0 - 2\d_0,
\ee
and it follows from \eqref{Compare} that
\bea
\nonumber
	d_o(\xi, \eta) &=& e^{- (\xi,\eta)_o} \leq 
	e^{-d(o, x) + \k} \leq e^{- l/2 + (\k + c_0 + 2\d_0)  } e^{-d(o,C)}.
\eea
For the claim, assume that $d(o, x) < d(o,C[s]) - 2\d_0$.
Consider the geodesic triangle given by $(o, x, \xi)$.
For any point $y$  on the ray $[o, \xi]$ at distance $d(o, y) > d(o,x) + \delta_0$ and $z$ on $[o,x]$, we have
$d(y,z) \geq d(o, y) - d(z, o) > \delta_0$.
Since $Z$ is a Gromov $\delta_0$-hyperbolic space, we see that $y$ cannot belong the $\delta_0$-neighborhood of the segment $[o, x]$
and must therefore be contained in $\cal{N}_{\d_0}([x, \xi])$.
Thus, let $x_s$ be a point on $\partial C[s] \cap [o, \xi]$ and note that $d(o,x_s) \geq d(o,C[s]) \geq d(o, x) + 2\d_0$.
From the above, we find a point $y$ on $[x, \xi]$ which is $\d_0$-close to $x_s$ and hence, $y \in C[s- \d_0]$.
Howeover, by convexity of the horoball $C[s-\d_0]$ and since $\eta = \partial_{\infty}C[s-\d_0]$, we have
\be
\nonumber
	[x, \eta] \subset [y, \eta] \subset C[s-\d_0].
\ee 
This shows $d(o,x) = d(o, [x, \eta]) \geq d(o,C[s-\d_0]) = d(o, C[s]) - \d_0$; a contradiction implying the claim.

Conversely, let $d_o(\xi, \eta) \leq c(l) e^{-d(o,C)}$ with $c(l) \leq \bar c e^{-l/2}$ where $\bar c>0$ is sufficiently large.
Set $t := D(o,C) + \delta_0$ and 
\be
\nonumber
	t+ \bar l \equiv -\log(d_o(\xi, C)) =  (\xi,\eta)_o 
	\geq \log(\bar c) + d(o,C) + l/2.
\ee
For $l>l_0=l_0(\bar c)$ sufficiently large, we have $t+ \bar l >t$.
Since $t+\bar l = (\xi, \eta)_o$, we have from \eqref{Tripod} that 
$d(\gamma_{o, \xi} (t+\bar l), \gamma_{o, \eta} (t+\bar l)) \leq  \delta_0$,
and by convexity also,
$d(\gamma_{o, \xi} (t), \gamma_{o, \eta} (t)) \leq  \delta_0$.
Thus, we obtain that
\be
\nonumber
	\gamma_{o, \xi} ([t, t+\bar l]) \subset 
	\cal{N}_{ \delta_0} ( \gamma_{o,\eta} ([d(o,C) + \delta_0, \infty))) \subset C.
\ee

The second case follows with similar arguments using Lemma \ref{LemmaClosing}.
The proof can be found in \cite{Parkkonen2}, Lemma 4.1.

For the third case, from Lemma 3.1 in \cite{DiophantineGeodesics}, 
there exist positive (universal) constants $c_1$, $c_2$, $c_3$  such that for all 
$x_m \in \H^{n+1}$, with  $d(o, c) \geq c_2$ (which we may assume if $t_0$ is sufficiently large), 
for all $0<R\leq c_3$ and $R\leq d(o, x_m)$ we have
\be
\nonumber
	B_{d_o}(\xi_m^{\infty}, Re^{-d(o,x_m)}) \subset \cal{S}_o(B(x_m,R)) 
	\subset B_{d_o}(\xi_m^{\infty}, c_1Re^{-d(o,x_m)}),
\ee
where $\cal{S}_o(B(x_m,R)) $ denotes the shadow at infinity of the metric ball $B(x_m, R)$ which is disjoint to $\{o\}$.
\end{proof}

\bibliographystyle         {plain}      
\bibliography{cup_ref}

 \end{document}